\documentclass[letterpaper, 10pt, journal, twocolumn, final]{IEEEtran}
\usepackage{cite}
\usepackage{amssymb,amsmath,amsfonts,amsthm}
\usepackage{graphicx}
\usepackage{textcomp}

\usepackage{mathtools}
\usepackage{bm}
\usepackage{algorithmicx}
\usepackage{algorithm}
\usepackage{algpseudocode}
\usepackage{graphicx}
\usepackage{epstopdf}
\usepackage{amsfonts}
\usepackage{dsfont}
\usepackage[hidelinks]{hyperref}
\usepackage{xcolor}

\usepackage{cuted}

\usepackage{enumitem}
\usepackage{tikz}
\usepackage{float}

\newtheorem{theorem}{Theorem}
\newtheorem{corollary}{Corollary}
\newtheorem{lemma}{Lemma}

\newtheorem{algo}{Algorithm}

\newtheorem{assumption}{Assumption}

\newcommand\oprocendsymbol{\hbox{$\square$}}
\newcommand\oprocend{\relax\ifmmode\else\unskip\hfill\fi\oprocendsymbol
}

\usepackage{tikz}
\makeatletter
\newcommand\mathcircled[1]{%
  \mathpalette\@mathcircled{#1}%
}
\newcommand\@mathcircled[2]{%
  \tikz[baseline=(math.base)] \node[draw,circle,inner sep=1pt] (math) {$\m@th#1#2$};%
}

\newcounter{AlphabetCounter}

\newcommand{\one}{\mathbf{1}}

\newcommand{\R}{\mathbb{R}}
\newcommand{\until}[1]{\{1,\ldots,#1\}}

\newcommand{\EE}{\mathcal{E}}
\newcommand{\GG}{\mathcal{G}}

\newcommand{\NN}{\mathcal{N}}

\newcommand{\VV}{\mathcal{V}}

\newcommand{\E}{\mathbb{E}}

\newcommand{\norm}[1]{\left\lVert#1\right\rVert}

\newcommand{\set}{\{1,\dots,N\}}
\newcommand{\col}{\text{col}}

\newcommand{\bL}{\bar{L}}

\newcommand{\x}{x}
\newcommand{\bx}{\mathbf{\x}}
\newcommand{\xstart}{\x_\star^{t+1}}
\newcommand{\xstartm}{\x_\star^{t}}
\newcommand{\xt}{\bx^{t+1}}
\newcommand{\xtm}{\bx^{t}}
\newcommand{\xit}{\x_{i}^{t+1}}
\newcommand{\xitm}{\x_{i}^{t}}

\newcommand{\xjtm}{\x_{j}^{t}}
\newcommand{\avgxt}{\bar{\x}^{t+1}}
\newcommand{\avgxtm}{\bar{\x}^{t}}

\newcommand{\s}{s}
\newcommand{\bs}{\mathbf{\s}}
\newcommand{\st}{\bs^{t+1}}
\newcommand{\stm}{\bs^{t}}
\newcommand{\sit}{\s_{i}^{t+1}}
\newcommand{\sitm}{\s_{i}^{t}}

\newcommand{\sjtm}{\s_{j}^{t}}
\newcommand{\avgst}{\bar{\s}^{t+1}}
\newcommand{\avgstm}{\bar{\s}^{t}}

\newcommand{\bg}{\mathbf{g}}
\newcommand{\gt}{\bg^{t+1}}
\newcommand{\gtm}{\bg^{t}}
\newcommand{\git}{g_{i}^{t+1}}
\newcommand{\gitm}{g_{i}^{t}}

\newcommand{\mm}{m}
\newcommand{\bmm}{\mathbf{\mm}}
\newcommand{\mt}{\bmm^{t+1}}
\newcommand{\mtm}{\bmm^{t}}
\newcommand{\mmit}{\mm_{i}^{t+1}}
\newcommand{\mmitm}{\mm_{i}^{t}}

\newcommand{\avgmt}{\bar{\mm}^{t+1}}
\newcommand{\avgmtm}{\bar{\mm}^{t}}

\newcommand{\vv}{v}
\newcommand{\bvv}{\mathbf{\vv}}
\newcommand{\vt}{\bvv^{t+1}}
\newcommand{\vtm}{\bvv^{t}}
\newcommand{\vit}{\vv_{i}^{t+1}}
\newcommand{\vitm}{\vv_{i}^{t}}

\newcommand{\avgvt}{\bar{\vv}^{t+1}}
\newcommand{\avgvtm}{\bar{\vv}^{t}}

\newcommand{\V}{V}
\newcommand{\bV}{\mathbf{\V}}
\newcommand{\Vt}{\bV^{t+1}}
\newcommand{\Vtm}{\bV^{t}}

\newcommand{\avgVtm}{\bar{\V}^{t}}

\newcommand{\dd}{d}
\newcommand{\bdd}{\mathbf{\dd}}
\newcommand{\dt}{\bdd^{t+1}}
\newcommand{\dtm}{\bdd^{t}}

\newcommand{\avgdt}{\bar{\dd}^{t+1}}
\newcommand{\avgdtm}{\bar{\dd}^{t}}

\newcommand{\ft}{f^{t+1}}
\newcommand{\ftm}{f^{t}}
\newcommand{\fit}{f_{i}^{t+1}}
\newcommand{\fitm}{f_{i}^{t}}
\newcommand{\nablaft}{\tfrac{1}{N}\sum_{i=1}^{N}\nabla \fit(\xit)}
\newcommand{\nablaftm}{\tfrac{1}{N}\sum_{i=1}^{N}\nabla \fitm(\xitm)}
\newcommand{\ones}{\tfrac{1}{N}\one\one^\top}
\newcommand{\eg}{e.g.}
\newcommand{\ie}{i.e.}

\allowdisplaybreaks

\def\algo/{GTAdam}
\def\GT/{GT}
\def\DGD/{DGD}
\def\DADAM/{DAdam}
\def\er/{Erd\H{o}s-R\'enyi}
\newcommand{\srN}{\sqrt{N}}
\newcommand{\ped}{\gamma}
\newcommand{\eig}{\chi}

\DeclareMathOperator{\diag}{diag}
\newcommand{\vertiii}[1]{{\left\vert\kern-0.25ex\left\vert\kern-0.25ex\left\vert #1 
    \right\vert\kern-0.25ex\right\vert\kern-0.25ex\right\vert}}

\begin{document}

\title{GTAdam: Gradient Tracking with Adaptive Momentum for Distributed Online Optimization}

\author{Guido Carnevale, Francesco Farina, Ivano
  Notarnicola, Giuseppe Notarstefano
  \thanks{The authors are with the Department of Electrical, 
  Electronic and Information Engineering, University of Bologna, Bologna, Italy,
  \texttt{\{name.lastname\}@unibo.it}.
  This result is part of a project that has received funding from the European 
  Research Council (ERC) under the European Union's Horizon 2020 research 
  and innovation programme (grant agreement No 638992 - OPT4SMART).
  }}

\maketitle

\begin{abstract}%
	This paper deals with a network of computing agents aiming to solve an online optimization problem in a distributed fashion, i.e., by means of local computation and communication, without any central coordinator. We propose the gradient tracking with adaptive momentum estimation (GTAdam) distributed algorithm, which combines a gradient tracking mechanism with first and second order momentum estimates of the gradient. The algorithm is analyzed in the online setting for strongly convex cost functions with Lipschitz continuous gradients. We provide an upper bound for the dynamic regret given by a term related to the initial conditions and another term related to the temporal variations of the objective functions. Moreover, a linear convergence rate is guaranteed in the static setup. The algorithm is tested on a time-varying classification problem, on a (moving) target localization problem, and in a stochastic optimization setup from image classification. In these numerical experiments from multi-agent learning, GTAdam outperforms state-of-the-art distributed optimization methods.
\end{abstract}

\section{Introduction}

In this paper, we deal with online optimization problems over networks and propose a new 
distributed algorithm. In this framework, interconnected computing agents have only a partial knowledge of the problem to solve, but can exchange information with neighbors according to a given communication graph and
without any central unit. 
In particular, we consider networks represented by a
weighted graph $\GG=(\VV,\EE,\mathcal{W})$, where $\VV=\until{N}$ is the set of agents,
$\EE\subseteq\VV\times\VV$ is the set of edges (or communication links), and
$\mathcal{W}\in\R^{N\times N}$ is the (weighted) adjacency matrix of the graph. The matrix
$\mathcal{W}$ is compliant with the topology described by $\EE$, \ie, being $w_{ij}$
the $(i, j)$-entry of $\mathcal{W}$, then $w_{ij} > 0$ if $(i,j) \in \mathcal{E}$ and
$w_{ij} =0$ otherwise. We denote $\NN_i = \{ j \in \VV \mid (j,i) \in \EE\}$
the set of (in-)neighbors of agent $i$. 

The aim of the network is to cooperatively solve the online optimization
problem
\begin{equation}\label{pb:problem}
	\min_{\x \in \R^n} \: \: \sum_{i=1}^N \fitm(\x), \quad t\geq 0,
\end{equation}
where each $\fitm:\R^n\to\R$ is a local function revealed only to agent $i$ at time
$t$.
In the following, we let $\ftm(\x) \triangleq \sum_{i=1}^N \fitm(\x)$. 
This distributed optimization framework captures a 
variety of estimation and learning problems over networks, including
distributed data classification and localization in smart sensor networks, see the recent survey~\cite{li2022survey} for an overview.

In this paper, we address the distributed solution of the online optimization
problem~\eqref{pb:problem} in terms of \emph{dynamic regret}
(see,~\eg,~\cite{li2022survey}).
In particular, let $\xitm$ be the solution estimate of the problem
at time $t$ maintained by agent $i$, and let
$\xstartm$ be a minimizer of $\sum_{i=1}^N \fitm$. 
Then, the agents want to minimize the dynamic regret defined as
\begin{align}\label{eq:regret}
	R_T \triangleq \sum_{t=1}^T\ftm(\avgxtm) -  \sum_{t=1}^T\ftm(\xstartm),
\end{align}  
for a finite value $T > 1$ with
$\avgxtm \triangleq \frac{1}{N}\sum_{i=1}^N \xitm$. Another possible performance metric is the so-called static regret (see, e.g.,~\cite{li2022survey}). %
The dynamic regret~\eqref{eq:regret} is known to be more challenging than the static one~\cite{li2022survey} and, for this reason, consistently with the majority of the recent papers in literature, this work focuses on the dynamic regret~\eqref{eq:regret}.
As it is customary in the distributed setting, we also complement these measures with the consensus metric $\sum_{i=1}^N \norm{\x_i^T - \bar{\x}^T}^2$,
 quantifying how far from consensus the local decisions are.

\subsubsection*{Related work}
\label{sec:related_work} %

The proposed distributed algorithm combines a gradient tracking mechanism 
with an adaptive estimation of first- and second-order momenta.

We organize the literature review in three main parts:
distributed algorithms for online optimization, gradient tracking distributed
schemes (mainly suited for static optimization), and centralized methods for
online and stochastic optimization based on adaptive momentum estimation.

Online optimization problems, characterized by time-varying cost functions, have
been originally addressed in the centralized framework, see, e.g.,
\cite{simonetto2016class,fazlyab2017prediction} and references therein, but
recently have received significant attention also in the distributed
optimization literature.
In~\cite{cavalcante2013distributed} an online optimization algorithm based on a
distributed subgradient scheme is proposed.
In \cite{towfic2014adaptive} an adaptive diffusion algorithm is proposed to
address changes regarding both the cost function and the constraints
characterizing the problem. A class of coordination algorithms that generalize
distributed online subgradient descent and saddle-point dynamics is proposed in
\cite{mateos2014distributed} for network scenarios modeled by jointly-connected
graphs. An algorithm consisting of a subgradient flow combined with a push-sum
consensus is studied in \cite{akbari2015distributed} for time-varying directed
graphs. Cost uncertainties and switching communication topologies are addressed
in \cite{hosseini2016online} by using a distributed algorithm based on dual
subgradient averaging. A distributed version of the mirror descent algorithm is
proposed in \cite{shahrampour2017distributed} to address online optimization
problems. In~\cite{akbari2019individual} an online algorithm based on the
alternating direction method of multipliers is proposed, and
in~\cite{yi2020distributed} time-varying inequality constraints are also
considered.
Online optimization is strictly related to stochastic optimization. Regarding
distributed algorithms for stochastic optimization, in \cite{ram2010distributed}
authors investigate the convergence properties of a distributed algorithm
dealing with subgradients affected by stochastic errors.
In \cite{farina2019randomized} a block-wise method is proposed to deal
with high-dimensional stochastic problems, while in~\cite{pu2020distributed} a
distributed gradient tracking method is analyzed in a stochastic set-up.

The gradient tracking scheme, which we extend in the present
paper, has been proposed in several variants in
recent years and studied under different problem
assumptions~\cite{shi2015extra,varagnolo2015newton,dilorenzo2016next,nedic2017achieving,
	qu_harnessing_2018,xu2017convergence,xin2018linear,scutari2019distributed}. %
This algorithm leverages a ``signal tracking action'' based on the dynamic average
consensus (see \cite{zhu2010discrete,kia2019tutorial}) in order to let the
agents obtain a local estimate of the gradient of the whole cost
function.
Recently, in~\cite{zhang2019distributed} the gradient tracking algorithm has
been applied to online optimization problems.
Finally, in~\cite{notarnicola2020personalized} a dynamic gradient tracking
update is combined with a recursive least squares scheme to address in a
distributed way the (centralized) personalized optimization framework introduced
in~\cite{simonetto2019personalized}.

The other algorithm inspiring our work is Adam, a centralized method originally
proposed in~\cite{kingma2014adam}. Adam is an optimization algorithm based on
adaptive estimates of first- and second-order gradient momenta that has been
successfully employed in many online and stochastic optimization frameworks.
Additional insights about Adam are given 
in~\cite{bock2018improvement,chen2018convergence,reddi2018convergence}, where 
some frameworks in which the algorithm is not able to reach the optimal 
solution are also shown.
This limitation is addressed in~\cite{zhou2018adashift}, where an effective extension of Adam, namely AdaShift, is proposed.
In~\cite{nazari2019dadam}, the authors proposed an enhanced version of the distributed gradient method with adaptive estimates of first- and second-order gradient momenta.
\subsubsection*{Contribution}
The main contribution of this paper is the design of a new distributed
algorithm to solve online optimization problems for multi-agent learning over networks. 
This novel scheme builds on the recently proposed gradient tracking distributed algorithm.
Specifically, in the gradient tracking the agents update their local solution estimates using a 
consensus averaging scheme perturbed with a local variable representing a descent direction.
This variable is concurrently updated using a dynamic consensus scheme aiming at reconstructing 
the total cost function gradient in a distributed way.
Inspired by the centralized Adam algorithm, we accelerate the basic gradient tracking scheme 
by enhancing the descent direction resorting to first- and second-order momenta of the 
cost function gradient. The use of momenta turned out to be very effective in the centralized Adam 
to solve online optimization problems with a fast rate. Therefore, we design our novel gradient tracking with adaptive momentum
estimation (\algo/) distributed algorithm to solve online optimization problems
over networks. The algorithm relies on local estimators for the two momenta, in
which the total gradient is replaced by a (local) gradient tracker.
Although the intuition behind the construction of \algo/ is clear and consists
of mimicking the centralized Adam in a distributed setting by using a gradient
tracking scheme, its analysis presents several additional challenges with
respect to both the gradient tracking and Adam. Indeed, being the descent
direction a nonlinear combination of the local states updated through a
consensus averaging, the proof approach of the gradient tracking needs to
be carefully reworked. 
We provide an upper bound about the dynamic regret for strongly convex online
optimization problems. This bound consists of a constant term, related to the
initial conditions of the algorithm, and another term depending on the temporal
variations of both the optimal solution of the problem and the gradients of the
objective functions. Thus, if the latter variations are sublinear with respect
to time, then our bound about the dynamic regret is sublinear too. A similar
result is also guaranteed for an agent-specific dynamic regret. Moreover, we
show that in the static case our algorithm reaches the optimal solution with a
linear rate.
Finally, we perform extensive numerical simulations on three application 
scenarios from distributed machine learning: a classification problem via logistic 
regression, a source localization problem in smart sensor networks and an 
image classification task.
We show that \algo/ outperforms in all cases the current state-of-the-art algorithms in terms of convergence rate.
\subsubsection*{Organization and Notation}
The paper is organized as follows. 
In Section~\ref{sec:existing_algorithms} we recall the two algorithms
that inspired the novel distributed algorithm proposed in this paper.  In
Section~\ref{sec:gtadam} \algo/ is presented with its convergence properties
which are proved in Section~\ref{sec:analysis}. Finally,
Section~\ref{sec:experiments} shows numerical examples highlighting the
advantages of \algo/.

The vertical concatenation of the vectors $v_1$ and $v_2$ is $\col (v_1, v_2)$. 
We use $\diag(v)$ to denote the diagonal matrix with diagonal
elements given by the components of $v$. 
The Hadamard product is denoted with $\odot$, while the Kronecker product with
$\otimes$. The identity
matrix in $\R^{m\times m}$ is $I_m$, while $0_m$ is the zero matrix in
$\R^{m\times m}$. The column vector of $N$ ones is denoted by $1_N$ and 
we define $\one \triangleq 1_N \otimes I_n$.
The spectral radius of a square matrix $M$ is denoted as $\rho(M)$.

\section{Inspiring algorithms}
\label{sec:existing_algorithms}
In this section we briefly recall two existing algorithms that represent the
building blocks for \algo/. 

\subsection{Adam centralized algorithm} Adam~\cite{kingma2014adam} is an
optimization algorithm that solves problems in the form~\eqref{pb:problem} in a
\emph{centralized} computation framework. It is an iterative gradient-like procedure
in which, at each iteration $t$, a solution estimate $\xtm$ is updated
by means of a descent direction which is enhanced by
a proper use of the gradient history, \ie, through estimates of their first- and second-order momenta.
Specifically, the (time-varying) gradient $\gtm = \nabla \ftm(\xtm)$ of the
function drives two exponential moving average estimators.  The two estimates,
denoted by $\mtm$ and $\vtm$, represent, respectively, mean and variance
($1^{st}$ and $2^{nd}$ momentum) of the gradient sequence and are nonlinearly
combined to build the descent direction.
A pseudo-code of Adam algorithm is reported in Algorithm~\ref{table:adam}
in which $\alpha>0$ is the step-size, the constant $0<\epsilon \ll 1$ is introduced to guarantee
numerical robustness of the scheme, while the hyper-parameters $\beta_1, \beta_2 \in (0, 1)$
control the exponential-decay rate of the moving average dynamics.
\begin{algorithm}%
		\begin{algorithmic}
			\State initialization: $x^{0}$ arbitrary, $m^{0} = v^{0}=0$, $g^0=\nabla f^{0}(\x^0)$
			\For{$t=1,2\dots$}
			\State $\mt = \beta_1 \mtm + (1-\beta_1) \gtm$
			\State $\vt  = \beta_2 \vtm + (1-\beta_2) \gtm\odot \gtm$
			\State $\xt  = \xtm - \alpha \frac{\sqrt{1-\beta_2}}{1-\beta_1} \frac{\mt}{\sqrt{\vt+\epsilon}}$
			\State $\gt = \nabla \ft(\xt)$
			\EndFor
		\end{algorithmic}
	\caption{Adam}
	\label{table:adam}
\end{algorithm}

We point out that in the algorithm above the
ratio $\frac{\mt}{\sqrt{\vt+\epsilon}}$ is meant element-wise. Typical choices
for the algorithmic parameters are $\beta_1=0.9$, $\beta_2=0.999$, and
$\epsilon=10^{-8}$.

\subsection{Gradient tracking distributed algorithm}

The gradient tracking is a distributed algorithm mainly tailored to
\emph{static} instances of problem~\eqref{pb:problem}.
Agents in a network maintain and update two local states $\xitm$ and $\sitm$
by iteratively combining a perturbed average consensus and a dynamic tracking
mechanism.  Consensus is used to enforce agreement among the local agents'
estimates $\xitm$. The agreement is also locally perturbed in order to steer
the local estimates toward a (static) optimal solution of the problem.  The
perturbation is obtained by using a tracking scheme that allows agents to
locally reconstruct a progressively accurate estimate of the whole gradient of
the (static) cost function in a distributed way. A pseudo-code of the gradient
tracking distributed algorithm is reported in Algorithm~\ref{table:GT}, in which $\NN_i$ denotes the set of (in-)neighbors of agent $i$, while
$\alpha>0$ is the step-size. The protocol is shown from
the perspective of agent $i$ only.
\begin{algorithm}%
		\begin{algorithmic}
			\State initialization: $x_{i}^{0}$ arbitrary, $s_{i}^{0} = g_{i}^{0} = \nabla f_{i}(x_{i}^{0}) $ %
			\For{$t=1,2\dots$}
			\State $\xit  = \sum\limits_{j\in\NN_i} w_{ij}\xjtm - \alpha\sitm$
			\State $\git = \nabla f_{i}(\xit)$
			\State $\sit  = \sum\limits_{j\in\NN_i} w_{ij}\sjtm  + \git - \gitm$
			\EndFor
		\end{algorithmic}
	\caption{Gradient tracking (for agent $i$)}
	\label{table:GT}
\end{algorithm}

\section{Gradient tracking with adaptive momentum estimation}
\label{sec:gtadam}

In this section we present the main contribution of this paper, \ie, the
gradient tracking with adaptive momentum estimation (\algo/) distributed
algorithm.
\algo/ is designed to address in a distributed fashion
problem~\eqref{pb:problem}, taking inspiration both from Adam and from the gradient
tracking distributed algorithm.

Along the evolution of the algorithm, each agent $i$ maintains four local states:
\begin{enumerate}[label=(\roman*)]
	\item a local estimate $\xitm$ of the current optimal solution $\xstartm$;
	\item an auxiliary variable $\sitm$ whose role is to track the gradient of the whole cost function;
	\item an estimate $\mmitm$ of the $1^{st}$ momentum of $\sitm$;
	\item an estimate $\vitm$ of the $2^{nd}$ momentum of $\sitm$.
\end{enumerate}
The momentum estimates of $\sitm$ are initialized as $m_i^0=v_i^0=0$, while 
the tracker of the gradient is initialized as $\s_{i}^{0}=\nabla f_{i}^{0} (\x_{i}^{0})$.

The algorithm works as follows. At each time instant $t$, each agent $i$ performs the following operations
\begin{enumerate}[label=(\roman*)]
	\item it updates the moving averages $\mmitm$ and $\vitm$;
	\item it computes a weighted average of the solution estimates of its neighbors
	and, starting from this point, it uses the update direction $\frac{\mmit}{\sqrt{\vit+\epsilon}}$ 
	to compute the new solution estimate $\xit$;
	\item it updates the local gradient tracker $\sitm$ via a ``dynamic consensus" mechanism.
\end{enumerate} 
A pseudo-code of \algo/ is reported in Algorithm~\ref{table:adamGT}. 
\begin{algorithm}%
		\begin{algorithmic}
			\State initialization: $x_{i}^{0}$ arbitrary, $s_{i}^{0} \!=\! g_{i}^{0} \!=\! \nabla f_{i}^{0}(\x_{i}^{0})$, $m_{i}^{0} \!=\! v_{i}^{0} \!=\! 0$ 
			\For{$t=1,\dots,T$}
			\State $\mmit = \beta_1 \mmitm + (1-\beta_1) \sitm$
			\State $\vit  = \min\{\beta_2 \vitm + (1-\beta_2) \sitm \odot \sitm,G\}$
			\State $\xit  = \sum\limits_{j\in\NN_i} w_{ij}\xjtm - \alpha\frac{\mmit}{\sqrt{\vit+\epsilon}}$\smallskip
			\State $\git = \nabla \fit(\xit)$
			\State $\sit  = \sum\limits_{j\in\NN_i} w_{ij}\sjtm  + \git - \gitm$
			\EndFor
		\end{algorithmic}
	\caption{\algo/ (for agent $i$)}
	\label{table:adamGT}
\end{algorithm}

Some remarks are in order. The algorithm proposed in this paper is different from~\cite{nazari2019dadam}. In fact, although they both use a similar strategy involving first- and second-order momenta, in that work only local gradients are considered, without resorting to any tracking mechanism.
Note that a saturation term $G \gg 0$ is introduced in the update of $\vitm$, where the $\min$ operator is to be intended element-wise.
The value of $G$ guarantees a bound for the scaling factor that multiplies the descent direction. 
Such a bound will turn out to be important for analysis purposes. 
We suggest to take it proportional to the initial estimates $v_i^0$.

We now state some regularity requirements on problem~\eqref{pb:problem}.
We first make two assumptions regarding each $\fitm$.
\begin{assumption}[Lipschitz continuous gradients]\label{assumption:lipschitz}
	The functions $\fitm$ have $L$-Lipschitz continuous gradients for all $i \in 
	\set$ and $t \ge 0$.%
\end{assumption}
\begin{assumption}[Strong convexity]\label{assumption:strong}
	The functions $\fitm$ are $\mu$-strongly convex for all $i \in \set$ and $t \ge 0$.%
\end{assumption}
We point out that, in light of Assumption~\ref{assumption:strong}, the minimizer 
$\xstartm$ is unique for all $t\ge 0$. 
Finally, the following characterizes the communication structure.
\begin{assumption}[Network Structure]\label{assumption:network}
	The weighted graph $\GG$ %
	is connected with doubly stochastic matrix
    $\mathcal{W}$ %
	stochastic. %
\end{assumption}
In order to analyze \algo/, we rewrite it into an aggregate form. Given
the variables $\{\xitm\}_{i=1}^N$, we define
$\xtm \triangleq \col(\x_{1}^{t}, \dots, \x_{N}^{t})$ and their average as
$\avgxtm \triangleq \frac{1}{N}\sum_{i=1}^N \xitm$. 
Similar definitions apply to the quantities $\mtm, \vtm, \dtm, \gtm, \stm$ and their averages
$\avgmtm, \avgvtm, \avgdtm, \avgstm$. With these definitions at hand, \algo/ can be rephrased from a global perspective as
\begin{subequations}\label{eq:global}
	\begin{align}
	\mt &= \beta_1 \mtm + (1-\beta_1) \stm\label{eq:m_global}\\
	\vt &= \min\{\beta_2 \vtm + (1-\beta_2) \stm \odot \stm,\one G\}\label{eq:v_global} \\
	\dt & = (\Vt + \epsilon I)^{-1/2} \mt%
	\label{eq:d_global} \\
	\xt &= W\xtm - \alpha\dt\label{eq:x_global}\\
	\st &= W\stm  + \gt - \gtm, \label{eq:s_global}
	\end{align}	
\end{subequations}
where we set $W \triangleq \mathcal{W} \otimes I_n$, $\Vtm \triangleq \text{diag}(\vtm)$, and $\avgVtm \triangleq \text{diag}(\avgvtm)$.
Moreover, the averaged quantities of~\eqref{eq:global} satisfy
\begin{subequations}\label{eq:avg}
	\begin{align}
	\avgmt &= \beta_1 \avgmtm + (1-\beta_1) \avgstm\label{eq:m_avg}
	\\
	\avgvt &= \min\{\beta_2 \avgvtm + (1-\beta_2) \avgstm \odot \avgstm,G\}\label{eq:v_avg} 
	\\
	\avgdt &= \tfrac{1}{N}\one^\top \dt \label{eq:d_avg}
	\\
	\avgxt &= \avgxtm - \alpha \avgdt\label{eq:x_avg}
	\\
	\avgst &= \avgstm  + \tfrac{1}{N}\sum_{i=1}^N (\git - \gitm). \label{eq:s_avg}
	\end{align}	
\end{subequations}
Our analysis is based on studying the aggregate dynamical evolution of the following: 
average first momentum $\| \avgmtm \|$, average tracking momentum difference $\|\avgstm \!-\! \avgmtm\|$, 
first momentum error $\|\mtm \!-\! \one\avgmtm\|$, gradient tracking error $\| \stm \!-\! \one\avgstm\|$, 
consensus error $\| \xtm \!-\! \one\avgxtm\|$ and solution error $\| \avgxtm \!-\! \xstartm\|$.
Let $y^t$ be the vector stacking the above quantities at iterations $t$
\begin{align}\label{eq:yt} 
y^t \triangleq 
\begin{bmatrix}
\|\avgmtm\|
\\
\|\avgstm - \avgmtm\|
\\
\|\mtm - \one\avgmtm\|
\\
\|\stm - \one \avgstm \|
\\
\|\xtm - \one \avgxtm\|
\\
\|\avgxtm - \xstartm\|
\end{bmatrix}.
\end{align}
	Notice that, due to the distributed context and no assumptions on the boundedness of the gradients, we need to take into account all these quantities to study the convergence.
	Let us introduce two useful variables that will be used to provide the main result of the paper, namely
	\begin{align}
	\begin{split}
	\eta^t &
	\triangleq \sup_i \sup_{x\in\R^{n}} \|\nabla \fit(x)\!-\!\nabla \fitm(x)\|,
	\\
	\zeta^t & \triangleq \|\xstart\!-\!\xstartm\|.
	\end{split}
	\label{eq:bound_optimum_gradient}
	\end{align}
Then, the main result of this paper is stated as follows.
\begin{theorem}\label{th:convergence}
	Consider \algo/ as given in Algorithm~\ref{table:adamGT}.
	Let Assumptions~\ref{assumption:lipschitz},~\ref{assumption:strong},
	and~\ref{assumption:network} hold.
	Then, for a sufficiently small step-size $\alpha>0$, there exists a constant $0 < \tilde{\rho} <1$, such that
	\begin{align}
	R_T \leq \frac{L\lambda^2}{2}\left(\frac{\norm{y^0}^2}{1 - \tilde{\rho}^2} + 2\norm{y^0}S_T + Q_T\right),\label{eq:result_1} 
	\end{align}
	where $R_T$ is defined in~\eqref{eq:regret}, the constant $\lambda$ is defined in the proof (cf.~\eqref{eq:lambda}) and 
	\begin{subequations}\label{eq:U_T_Q_T}
		\begin{align}
			S_T &\triangleq \sum_{t=1}^T\sum_{k=0}^{t-1}\tilde{\rho}^{t+k}\left(\frac{N+1}{\sqrt{N}}\norm{\eta^{t-k-1}} + \norm{\zeta^{t-k-1}}\right) 
			\\
			Q_T &\triangleq \sum_{t=1}^T\left(\sum_{k=0}^{t-1}\tilde{\rho}^{k}\frac{N+1}{\srN}\norm{\eta^{t-k-1}} + \norm{\zeta^{t-k-1}}\right)^2,
		\end{align}
	\end{subequations}
	where $\eta^t, \zeta^t$ are defined in~\eqref{eq:bound_optimum_gradient} and we assume that are finite.
Moreover, it holds %
\begin{align}
\lim_{T\to\infty} \sum_{i=1}^N \norm{\x_i^T - \bar{\x}^T}^2\leq \tfrac{\lambda^2}{(1-\tilde{\rho})^2}\max_t\left\{\tfrac{N^2 + 1}{N}\eta^t + \zeta^t\right\}.\label{eq:result_3}%
\end{align}
\end{theorem}

As it requires several intermediate results, the proof of Theorem~\ref{th:convergence} is carried out in Section~\ref{sec:analysis}.

	There is evidence in the literature, see, e.g.,~\cite{shahrampour2017distributed,mokhtari2016online,dall2020optimization,li2020online,notarnicola2020personalized,li2022survey}, that the bound on the dynamic regret cannot be sublinear with respect to $T$. As stated, e.g., in~\cite{li2022survey}, when the objective functions are strongly convex and have bounded gradients, the bound on dynamic regret is $O(1 + \eta^t)$. Our work does not assume gradient boundedness and, thus, our bound has additional terms due to variations over time of the gradients. Specifically, Theorem~\ref{th:convergence} shows that $R_T$ is upper bounded by a constant depending on the initial conditions and by other two terms.
	The latters involve $S_T$ and $Q_T$, which capture the time-varying nature of the problem itself. Indeed, suppose that the problem varies linearly, i.e., there exists $C > 0$ so that $\eta^t, \zeta^t \leq C$ for all $t \ge 0$. Then, being $\tilde{\rho} \in (0,1)$, we can exploit the geometric series properties to write the following 
\begin{align*}
	S_T &\leq \frac{(N+\sqrt{N} +1)(\tilde{\rho} - \tilde{\rho}^{T+1})(1 - \tilde{\rho}^T)C}{\sqrt{N}(1 -\tilde{\rho})^2}
	\\
	Q_T &\leq \frac{(N + \srN +1)^2(1-\tilde{\rho}^{T})^2C^2T}{N(1 - \tilde{\rho})^2}.
\end{align*}
In this case,~\eqref{eq:result_1} ensures that the average regret $R_T/T$ asymptotically approaches a constant when $T \to \infty$, specifically
\begin{align*}
	\lim_{T \to \infty} \frac{R_T}{T} \leq \frac{L\lambda^2(N^2 + \srN +1)^2C^2}{2N(1 - \tilde{\rho})^2}.
\end{align*}
The key point of the proof consists in showing that the error vector $y^t$ (see~\eqref{eq:yt}) evolves according to a linear system with state matrix $A(\alpha)$ (whose entries depend on the problem parameters such, e.g., the strong convexity function or the network connectivity) which is perturbed by an input $q^t$ related to the variations of the problem over time (see~\eqref{eq:linear_system}). Notice that the parameter $\tilde{\rho}$ is related to the spectral radius of $A(\alpha)$ and, thus, depends also on the network topology.

\subsubsection*{Agent Regret}
We may also consider a regret for each agent $i$ defined as
  $R_{T,i} \triangleq\sum_{t=1}^T \ftm(\xitm) -\sum_{t=1}^T\ftm(\xstartm)$.
\begin{corollary}\label{cor:agent_regret}
  Under the same assumptions of Theorem~\ref{th:convergence}, for all $i \in \set$, it holds
	\begin{align*}
	R_{T,i}	&\leq 2L\lambda^2\left(\frac{\norm{y^0}^2}{1-\tilde{\rho}^2} +  2\norm{y^0}S_T + Q_T\right),
	\end{align*}
	where $\lambda$, $\tilde{\rho}$, $S_T$, and $Q_T$ are defined as in Theorem~\ref{th:convergence}.%
\end{corollary}
The proof is given in Appendix~\ref{sec:proof_cor_agent_regret}.

\subsubsection*{Static set-up}
We provide an additional corollary of Theorem~\ref{th:convergence} asserting theoretical guarantees in a static scenario. Specifically, for this special case the \algo/ distributed algorithm converges to the optimal solution with a linear rate.
	\begin{corollary}[Static set-up]\label{cor:static}
		Under the same assumptions of Theorem~\ref{th:convergence}, if additionally holds $\ftm = f$ for all $t \ge 0$, then, for a sufficiently small step-size $\alpha > 0$, there exists a constant $0 < \tilde{\rho} < 1$ such that
		\begin{equation}
		f(\avgxtm) - f(\xstartm) \leq \tilde{\rho}^{2t}\frac{L\lambda^2}{2}\norm{y^0}^2, 
		\label{eq:static_linear_convergence}
		\end{equation}
		where the constant $\lambda$ is defined in~\eqref{eq:lambda}.%
	\end{corollary}
	
	The proof is given in Appendix~\ref{sec:proof_cor_static}.

\section{Analysis}\label{sec:analysis}
This section is devoted to provide the proof of Theorem~\ref{th:convergence}.

\subsection{Preparatory Lemmas}
We now give a sequence of intermediate results, providing proper bounds
on the components of $y^t$ (defined in~\eqref{eq:yt}), that are then used as building blocks for proving
Theorem~\ref{th:convergence}. %
\begin{lemma}[Average first momentum magnitude]\label{lemma:m_bound}
	Let Assumption~\ref{assumption:lipschitz} holds. Then, for all $t\geq 0$, it holds 
	\begin{align}
	\norm{\avgmt} 
	& \! \leq \!
	\beta_1 \norm{\avgmtm} 
	\!+\!
	\tfrac{(1\!-\!\beta_1)L}{\srN}\! \norm{\xtm \!-\! \one \avgxtm}
	\!+\! (1\!-\!\beta_1)L\! \norm{\avgxtm \!-\! \xstartm}\!.\notag%
	\end{align}
\end{lemma}

The proof is given in Appendix~\ref{sec:proof_lemma_1}.
\begin{lemma}[First momentum error]\label{lemma:m}
	For all $t\geq 0$, it holds 
	\begin{align*}
	\norm{\mt - \one\avgmt} & \leq \beta_1\norm{ \mtm - \one\avgmtm} + (1-\beta_1)\norm{ \stm -\one \avgstm}.%
	\end{align*}
\end{lemma}

The proof of Lemma~\ref{lemma:m} follows by combining~\eqref{eq:m_global} and~\eqref{eq:m_avg} with the triangle inequality.
\begin{lemma}[Input signal error]\label{lemma:d}
	For all $t\geq 0$, it holds 
	\begin{align*}
	& \|\dt - \one \avgdt \| 
	\\
	& \leq \tfrac{\beta_1\srN}{\sqrt{\epsilon}}\norm{\avgmtm} +\tfrac{\beta_1}{\sqrt{\epsilon}}\norm{\mtm -\one\avgmtm} 
	+\tfrac{(1-\beta_1)}{\sqrt{\epsilon}}\norm{\stm - \avgstm} 
	\\
	&\hspace{3ex} +  \tfrac{(1-\beta_1)L}{\sqrt{\epsilon}}\norm{\xtm - \one \avgxtm} 
	+ \tfrac{(1-\beta_1)\beta_1L\srN}{\sqrt{\epsilon}}\norm{\avgxtm - \xstartm}.%
	\end{align*}
\end{lemma}

The proof is given in Appendix~\ref{sec:proof_lemma_3}.
\begin{lemma}[Tracking error]\label{lemma:s}
	Let Assumptions~\ref{assumption:lipschitz},~\ref{assumption:strong},	and~\ref{assumption:network} hold.
	Then, for all $t\geq 0$, it holds 
	\begin{align*}
	&\|\st - \one \avgst \|\leq \left(\sigma_W + \alpha\tfrac{ 2(1-\beta_1)L}{\sqrt{\epsilon}}\right)\norm{\stm - \one\avgstm} 
	\\
	&\hspace{3ex}+ \alpha\tfrac{ 2 \beta_1L\srN}{\sqrt{\epsilon}}\norm{\avgmtm}+\alpha\tfrac{ 2\beta_1L}{\sqrt{\epsilon}}\norm{ \mtm - \one\avgmtm}
	\\
	&\hspace{3ex}+\left(\!L\norm{W\!-\!I} + \alpha\tfrac{ 2(1-\beta_1)\beta_1L^2}{\sqrt{\epsilon}} \right)\norm{\xtm-\one\avgxtm} 
		\notag\\
	&\hspace{3ex}+\alpha\tfrac{ (1-\beta_1)(1+\beta_1)L^2\srN}{\sqrt{\epsilon}}\norm{\avgxtm - \xstartm}	+ \srN\eta^t.
	\end{align*}%
	where $\sigma_W \in (0,1)$ is the spectral radius of $W-\ones$ and $\eta^t$ has been defined in~\eqref{eq:bound_optimum_gradient}.%
\end{lemma}

The proof is given in Appendix~\ref{sec:proof_lemma_4}.

\begin{lemma}[Consensus error]\label{lemma:x}
	Let Assumptions~\ref{assumption:lipschitz}, and~\ref{assumption:network} hold. 
	Then, for all $t\geq 0$, it holds 
	\begin{align*}
	&\|\xt - \one \avgxt\| \leq \left(\sigma_W + \alpha\tfrac{(1-\beta_1)L}{\sqrt{\epsilon}}\right)\|\xtm - \one \avgxtm\| 
	\\	
&\hspace{3ex}+\alpha\tfrac{\beta_1\srN}{\sqrt{\epsilon}}\norm{\avgmtm} +\alpha\tfrac{\beta_1}{\sqrt{\epsilon}}\norm{\mtm -\one\avgmtm} 
	\\
	&\hspace{3ex} +\alpha\tfrac{(1-\beta_1)}{\sqrt{\epsilon}}\norm{\stm - \avgstm} 
+ \alpha\tfrac{(1-\beta_1)\beta_1L\srN}{\sqrt{\epsilon}}\norm{\avgxtm - \xstartm}.%
	\end{align*}
\end{lemma}

The proof is given in Appendix~\ref{sec:proof_lemma_5}.

\begin{lemma}[Tracking momentum difference magnitude]\label{lemma:z}
	Let Assumptions~\ref{assumption:lipschitz}, \ref{assumption:strong},	and \ref{assumption:network} hold. Then, for all $t\geq 0$, it holds 
	\begin{align*}
	\|\avgst & - \avgmt\|
	\leq\beta_1\|\avgstm - \avgmtm\| + \alpha\tfrac{ \beta_1L}{\sqrt{\epsilon}}\norm{\avgmtm}
	\\
	&\hspace{3ex} +   \alpha\tfrac{2\beta_1L}{\sqrt{\epsilon}\srN}\| \mtm - \one\avgmtm \| + \alpha\tfrac{ L}{\sqrt{\epsilon}\srN}\norm{\stm - \one\avgstm}
	\\
	&\hspace{3ex}+\left(\sigma_W\frac{L}{\srN} +\tfrac{L}{\srN} + \alpha\tfrac{(1-\beta_1)L^2}{\sqrt{\epsilon}\srN}\right)\norm{\xtm - \one \avgxtm}
	\\
	&\hspace{3ex}+ \alpha\tfrac{(1-\beta_1)L^2}{\sqrt{\epsilon}}\norm{\avgxtm - \xstartm}+\tfrac{1}{\srN}\eta^{t}.%
	\end{align*}
\end{lemma}

The proof is given in Appendix~\ref{sec:proof_lemma_6}.

\begin{lemma}[Solution error]\label{lemma:xbar}
	Let Assumptions~\ref{assumption:lipschitz},
	\ref{assumption:strong}, and \ref{assumption:network} hold. Then, for all
	$t\geq 1$, it holds
	\begin{align*}
	& \|\avgxt - \xstart\|
	\\
	&
	\leq (1-\alpha\delta)\|\avgxtm  \!-\!  \xstartm\| + \alpha\tfrac{\beta_1}{\sqrt{\epsilon}}\|\avgstm  \!-\! \avgmtm\|
	+ \alpha\tfrac{ L}{\sqrt{\epsilon}\srN}\|\xtm \!-\! \one\avgxtm\| 
	\\
	&\hspace{3ex} +\alpha\tfrac{\beta_1}{\sqrt{\epsilon}\srN}\norm{\mtm-\one\avgmtm} 
	+ \alpha\tfrac{(1-\beta_1)}{\sqrt{\epsilon}\srN}\norm{\stm-\one\avgstm} + \zeta^t,
	\end{align*}
	where $\zeta^t$ is defined in~\eqref{eq:bound_optimum_gradient} and
$\delta \triangleq \min\left\{\tfrac{\mu}{\sqrt{\epsilon+G}}, \tfrac{L}{\sqrt{\epsilon}}\right\}$.%
\end{lemma}

The proof is given in Appendix~\ref{sec:proof_lemma_7}.

\subsection{Proof of Theorem 1}
\label{sec:proof}

By recalling the definition of $y^t$ given in~\eqref{eq:yt}
and combining Lemma~\ref{lemma:m_bound},~\ref{lemma:m},~\ref{lemma:s},~\ref{lemma:x}, \ref{lemma:z}, \ref{lemma:xbar}, it is possible to write
\begin{align}\label{eq:linear_system}
	y^{t+1} \leq A(\alpha) y^{t} + q^t,
\end{align}
where 
$q^t \triangleq \col\left(0, \frac{1}{\sqrt{N}}\eta^{t}, 0, \sqrt{N}\eta^t, 0, \zeta^t\right)$.
The matrix $A(\alpha)$ can be decomposed in
$A(\alpha)  \triangleq A_0 + \alpha E$,
with
\begin{align*}
A_0 &\triangleq \begin{bmatrix}	
\beta_1& 0& 0& 0& \beta_1c_1& (1-\beta_1)L\\
0& \beta_1& 0& 0& \sigma_Wc_1 + c_1& 0\\
0& 0&\beta_1& 1-\beta_1& 0& 0\\
0&0&0&\sigma_W&c_2& 0\\
0& 0&0& 0& \sigma_W& 0\\
0& 0&0& 0& 0& 1
\end{bmatrix}
\end{align*}
and
\begin{align*}
\!E\!\!\triangleq\!\!\begin{bmatrix}	
\!\!\!0& 0& 0& 0& 0& 0\!\!\!\!
\\
\!\!\tfrac{\beta_1L}{\sqrt{\epsilon}} &0& \tfrac{2\beta_1c_1}{\sqrt{\epsilon}} &\tfrac{c_1}{\sqrt{\epsilon}} &\tfrac{(1-\beta_1)c_1L}{\sqrt{\epsilon}}& \tfrac{(1-\beta_1)L^2}{\sqrt{\epsilon}}\!\!\!\!
\\
\!\!0& 0&0& 0& 0& 0\!\!\!\!
\\
\!\!\tfrac{2\beta_1L\srN}{\sqrt{\epsilon}}&0&\tfrac{2\beta_1L}{\sqrt{\epsilon}} &  \tfrac{2(1-\beta_1)L}{\sqrt{\epsilon}}& c_3& c_4\!\!\!\!
\\
\!\!\tfrac{\beta_1\srN}{\sqrt{\epsilon}}& 0& \tfrac{\beta_1}{\sqrt{\epsilon}}& \tfrac{1-\beta_1}{\sqrt{\epsilon}}& 0& c_5\!\!\!\!
\\
\!\!0&\tfrac{\beta_1}{\sqrt{\epsilon}}& \tfrac{\beta_1}{\sqrt{\epsilon}\srN}& 0& \tfrac{(1-\beta_1)}{\sqrt{\epsilon}\srN}& -\delta&
\!\!\!\!\end{bmatrix}\!\!,
\end{align*}%
where we used the following shorthands
\begin{align*}
 c_1 &\triangleq \tfrac{L}{\srN}, \quad c_2 \triangleq L\|W-I\|, \quad c_3 \triangleq \tfrac{2(1-\beta_1)\beta_1L^2}{\sqrt{\epsilon}},
 \\
  c_4 &\triangleq \tfrac{(1-\beta_1)(1+\beta_1)L^2\sqrt{N}}{\sqrt{\epsilon}}, \quad c_5 \triangleq \tfrac{(1-\beta_1)\beta_1L\srN}{\sqrt{\epsilon}}.
\end{align*}%
Being $A_0$ triangular, it is easy to see that its spectral radius is $1$ since both $\beta_1$ and $\sigma_W$ are in $(0, 1)$. We want to study how the perturbation matrix $\alpha E$ affects the simple eigenvalue $1$ of $A_0$. Hence, we denote by $\eig(\alpha)$ such eigenvalue of $A(\alpha)$ as a function of $\alpha$. Call $w$ and $v$ respectively the left and right eigenvectors of $A_0$ associated to the eigenvalue $1$, then
	$ w=\col\left(0, 0, 0, 0, 0, 1\right)$
	and
	$v=\col\left(L, 0, 0, 0, 0, 1\right)$.
Since the eigenvalue $1$ is simple, from~\cite[Theorem~6.3.12]{horn2012matrix} it holds
\begin{align*}
\frac{d\eig(\alpha)}{d\alpha}\bigg|_{\alpha=0} = \frac{w^\top Ev}{w^\top v} = -\delta < 0.
\end{align*}
Then, by continuity of eigenvalues with respect to the matrix entries, $\eig(\alpha)$ is strictly less than $1$ for sufficiently small $\alpha>0$. 
Then, it is always possible to choose $\alpha>0$ so as the remaining eigenvalues stay in the unit circle. Therefore, the spectral radius is $\rho(A(\alpha)) < 1$. Moreover, since $A(\alpha)$ and $q^t$ have only non-negative entries, one can use~\eqref{eq:linear_system} to write
\begin{equation}\label{eq:y_evolution}
	y^{t} \leq A(\alpha)^{t} y^0 + \sum_{k=0}^{t-1} A(\alpha)^{t-1-k}q^{k}.
\end{equation}
From~\cite[Lemma~5.6.10]{horn2012matrix}, we have that for any $\gamma > 0$, there exists a matrix norm, say $\vertiii{\cdot}_\ped$, such that
\begin{align}\label{eq:inequaility_horn}
	\vertiii{A(\alpha)}_\ped \leq \rho(A(\alpha)) + \gamma.
\end{align}
Let us pick $\gamma \in (0, 1 - \rho(A(\alpha)))$ and define $\tilde{\rho} \triangleq \rho(A(\alpha)) + \gamma$.
Then, in light of~\eqref{eq:inequaility_horn} it holds $\vertiii{ A(\alpha)}_\ped \leq \tilde{\rho} < 1$. 
Moreover, by applying~\cite[Theorem~5.7.13]{horn2012matrix}, there exists a vector norm $\| \cdot\|_\ped$ 
such that $\|Mv\|_\ped \leq \vertiii{M}_\ped\|v\|_\ped$ for any matrix $M \in \R^{6 \times 6}$ and $v \in \R^6$. 
Hence, we can manipulate~\eqref{eq:y_evolution} taking the norm and using the triangle inequality to write
\begin{align}\notag
	\norm{y^{t}}_\ped 
	& \leq 
	\norm{A(\alpha)^{t}y^0}_\ped + \norm{\sum_{k=0}^{t-1} A(\alpha)^{t-1-k}q^{k}}_\ped
	\\
	& \leq \tilde{\rho}^{t}\norm{y^0}_\ped + \sum_{k=0}^{t-1}\tilde{\rho}^{k}\norm{q^{t-1-k}}_\ped,
	\label{eq:bound_y}
\end{align}
which shows that first term decreases linearly with rate $\tilde{\rho} < 1$ while the second one is bounded. By using the Lipschitz continuity of the gradients of $\ftm$ (cf. Assumption~\ref{assumption:lipschitz}), we have
\begin{align}
f^t(\avgxtm) - f^t(\xstartm) \leq \tfrac{L}{2}\|\avgxtm - \xstartm\|^2 \stackrel{(a)}{\leq} \frac{L}{2}\norm{y^t}^2,\label{eq:f_minus_fstar}
\end{align}
where in $(a)$ we use the fact that $\norm{\avgxtm - \xstartm}$ represents a component of $y^t$ 
leading to the trivial bound $\norm{\avgxtm - \xstartm} \leq \norm{y^t}$. Recalling that all norms are equivalent on finite-dimensional vector spaces, there always exist $\lambda_1 > 0$ and $\lambda_2 > 0$ such that%
\begin{subequations}%
\begin{align}%
	\norm{\cdot} &\leq \lambda_1\norm{\cdot}_\ped\label{eq:lambda_1}
	\\
	\norm{\cdot}_\ped &\leq \lambda_2\norm{\cdot}.\label{eq:lambda_2}
\end{align}
\end{subequations}
Thus, by applying~\eqref{eq:lambda_1}, we bound~\eqref{eq:f_minus_fstar} as
\begin{align*}
	f^t(\avgxtm) - f^t(\xstartm) \leq \frac{L\lambda_1}{2}\norm{y^t}^2_\ped,
\end{align*}
which, combined with the definition of $R_T$ (cf.~\eqref{eq:regret}) and the result~\eqref{eq:bound_y}, leads to 
\begin{align}
	R_T &\leq  \frac{L\lambda_1^2}{2}\!\bigg(\!\sum_{t=1}^T  \tilde{\rho}^{2t}\norm{y^0}_\ped^2 + 2\!\norm{y^0}_\ped\sum_{t=1}^T\sum_{k=0}^{t-1}\tilde{\rho}^{t+k}\norm{q^{t-1-k}}_\ped  
	\notag\\
	&\hspace{3.8cm}
	+ \sum_{t=1}^T\bigg(\sum_{k=0}^{t-1}\tilde{\rho}^{k}\norm{q^{t-1-k}}_\ped\bigg)^2\! \bigg)
	\notag\\
	&\stackrel{(a)}{\leq} \frac{L\lambda_1^2\lambda_2^2}{2}\!\bigg(\!\frac{\norm{y^0}^2}{1 - \tilde{\rho}^2} + 2\!\norm{y^0}\sum_{t=1}^T\sum_{k=0}^{t-1}\tilde{\rho}^{t+k}\norm{q^{t-1-k}}  
	\notag\\
	&\hspace{3.3cm}
	+ \sum_{t=1}^T\bigg(\sum_{k=0}^{t-1}\tilde{\rho}^{k}\norm{q^{t-1-k}}\bigg)^2 \!\bigg),
\end{align}
where in $(a)$ we use the geometric property series and the relation~\eqref{eq:lambda_2}. 
The proof follows by using the triangle inequality, the definitions of $U_T$ and $Q_T$ (cf.~\eqref{eq:U_T_Q_T}), and by setting
\begin{align}
	\lambda \triangleq \lambda_1\lambda_2.\label{eq:lambda}
\end{align}
Finally, in order to prove~\eqref{eq:result_3}, we notice that 
$	 \sum_{i=1}^N \norm{\x_i^T - \bar{\x}^T}^2 \leq \norm{y^T}^2
	\leq \lambda_1^2\norm{y^T}_\ped^2$,
in which we apply~\eqref{eq:lambda_1}. 
By applying the bound~\eqref{eq:bound_y} for $t = T$, we get
\begin{align*}
	\norm{y^{T}}_\ped &\leq \tilde{\rho}^{T}\norm{y^0}_\ped + \sum_{k=0}^{T-1}\tilde{\rho}^{k}\norm{q^{T-k-1}}_\ped.
\end{align*}
The first term of the latter inequality vanishes as $T \to \infty$, while the second one can be bounded by relying on geometric series property and $\max_k\{\norm{q^k}^2\}$. By exploiting these arguments, we can write
\begin{align}
	\hspace{-0.2cm}
	\lim_{T \to \infty} \sum_{i=1}^N \norm{\x_i^T - \bar{\x}^T}^2
	&
	\!\leq\! \tfrac{\lambda_1^2}{(1 \!-\! \tilde{\rho})^2}\max_{t}\left\{\norm{q^{t}}_\ped^2\right\}
 \notag\\
 &
\!\stackrel{(a)}{\leq}\! \tfrac{\lambda^2}{(1 \!-\! \tilde{\rho})^2}\max_{t}\left\{\norm{q^{t}}^2\right\}\!\!.
\label{eq:lim_C}
\end{align}
where in $(a)$ we apply~\eqref{eq:lambda_2} and the definition~\eqref{eq:lambda} of $\lambda$. 
The result~\eqref{eq:result_3} follows by noticing that 
\begin{align*}
	\max_{t}\{\norm{q^{t}}^2\} 
	\!=\! 
	\max_t\Big\{\!\tfrac{N^2 + 1}{N}\eta^t + \zeta^t\Big\}.
\end{align*}

\section{Numerical Experiments}
\label{sec:experiments}

In this section we consider three multi-agent distributed learning problems to show 
the effectiveness of \algo/.
The first scenario regards the computation of a linear classifier via a regularized logistic 
regression function for a set of points that change over time.
The second scenario involves the localization of a moving target.
The third example is a stochastic optimization problem arising in a distributed 
image classification task. In all the examples, the parameters of \algo/ are chosen as $\beta_1=0.9$, $\beta_2=0.999$, and 
$\epsilon=10^{-8}$. %
Moreover, we compare \algo/ with the gradient tracking distributed algorithm (\GT/) (cf.~Algorithm~\ref{table:GT}
in Section~\ref{sec:existing_algorithms}), the distributed gradient descent
(\DGD/) (see~\cite{nedic2009distributed}), and the distributed Adam (\DADAM/) (see~\cite{nazari2019dadam}) described by
\begin{align*}
	\mmit &= \beta_1 \mmitm + (1-\beta_1)\nabla\fitm(\xitm)
	\\
	\vit &= \beta_2 \vitm + (1 - \beta_2)\nabla\fitm(\xit)\odot\nabla\fitm(\xit)
	\\
	\tilde{v}_i^{t+1} &= \beta_3 \tilde{v}_i^{t} + (1 - \beta_3)\max\{\tilde{v}_i^{t},\vit\}
	\\
	\xit &= \sum_{j \in \NN_i}w_{ij}\xjtm + \gamma^t \frac{\mmit}{\tilde{v}_i^{t+1}},
\end{align*}
for all $i \in \until{N}$. As suggested in~\cite{nazari2019dadam},  we set $\beta_1 = \beta_3 = 0.9$, $\beta_2 = 0.999$, and a diminishing stepsize $\gamma^t = (\frac{\alpha}{t})^{-1/2}$ for some $\alpha > 0$.

\subsection{Distributed classification via logistic regression}
\label{sec:classification}

Consider a network of agents that want to cooperatively %
train a linear classifier for a set of (moving) points in a given feature space. At time $t \ge 0$, each agent $i$ is equipped 
with $m_i \in \mathbb{N}$ points $p_{i,1}^t, \dots, p_{i,m_i}^t \in \R^d$ with binary 
labels $l_{i,k} \in \{-1,1\}$ for all $k \in \until{m_i}$. 
The problem consists of building a linear classification model from the given points, also
called training samples. In particular, we look for a separating hyperplane described
by a pair $(w,b) \in\R^{d}\times \R$ given by
$\{p \in \R^d \mid w^\top p + b = 0\}$. 
This \emph{online} classification problem can be posed at each time $t \ge0$, as a minimization problem described by
\begin{equation}\label{eq:logistic_regression}
	\min_{w,b}\:
	\sum_{i=1}^{N}\sum_{k=1}^{m_i}\log\left(1 + e^{-l_{i,k}(w^\top p_{i,k}^t + b)}\right) 
	+ \tfrac{C}{2}\left(\|w\|^2 + b^2 \right),
\end{equation}
where $C > 0$ is the so-called regularization parameter. Notice that the 
presence of the regularization makes the cost function strongly convex. 
Each point $p_{i,k}^t \in \R^2$ moves along a circle of radius $r = 1$ according to the following law
 \begin{equation*}
p_{i,k}^t = p_{i,k}^c + r
\begin{bmatrix}
\cos(t/100)
\\
\sin(t/100)
\end{bmatrix},
 \end{equation*}
where $p_{i,k}^c \in \R^2$ represents the randomly generated center of the considered circle.  
We consider a network of $N=50$ agents and pick $m _i = 5$ (for all $i$).
We performed an experimental tuning to optimize the step-sizes to enhance
the convergence properties of each algorithm. In particular, we selected
$\alpha = 0.1$ for \algo/, $\alpha = 0.05$ for \GT/, $\alpha = 0.1$ for
\DGD/, and $\alpha = 0.1$ for \DADAM/.
We performed Monte Carlo simulations consisting of $100$ trials, in which we alternatively consider an undirected, connected \er/ graph with connectivity parameter $0.5$, and a ring graph.
In Fig.~\ref{fig:logistic_montecarlo_relative_cost}, we plot the average across
the trials of the relative cost error, namely
$\frac{f^t(\avgxtm) - f^t(\xstartm)}{f^t(\xstartm)}$, with $\xstartm$
being the minimum of $f^t$ for all $t$.
\begin{figure}
	\centering
	\includegraphics[scale=0.9]{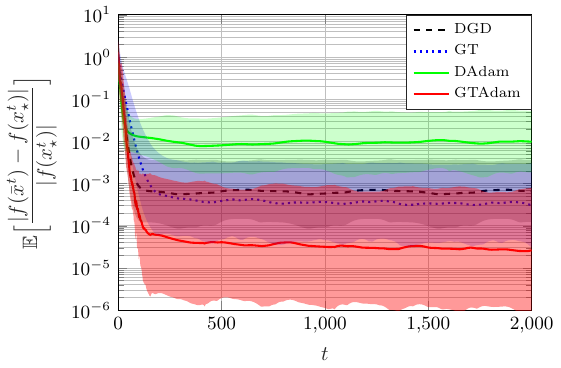}
	\caption{Distributed classification via logistic regression. Mean of the relative cost errors and $1$-standard deviation band obtained 
		with Monte Carlo simulations consisting of $100$ trials in which each of the $N=50$ agents 
		is equipped with $m = 5$ points.}
	\label{fig:logistic_montecarlo_relative_cost}
\end{figure}
The plot highlights that \algo/ exhibits a faster convergence compared to the
other algorithms, and achieves a smaller tracking error.

	Finally, we consider a static instance of problem~\eqref{eq:logistic_regression}, i.e., with fixed objective function $\fitm = f_i$ for all $t \ge 0$ and $i \in \until{N}$. We consider a network of $N=50$ agents in a ring topology. We take $\alpha = 0.001$ for \algo/, $\alpha = 0.01$ for \GT/, $\alpha = 0.1$ for \DGD/, and $\alpha = 0.5$ for \DADAM/. 
	In Fig.~\ref{fig:static}, we plot the error $\norm{\avgxtm - x_\star}$ achieved by the considered methods, where $x_\star \in \R^d$ is the (fixed) optimal solution of the problem. Fig.~\ref{fig:static} clearly shows the benefit of the tracking mechanism, which allows \algo/ and \GT/ to achieve the exact problem solution. The plot also shows that \algo/ is faster than \GT/.
	\begin{figure}[!htpb]
		\centering
		\includegraphics[scale=0.9]{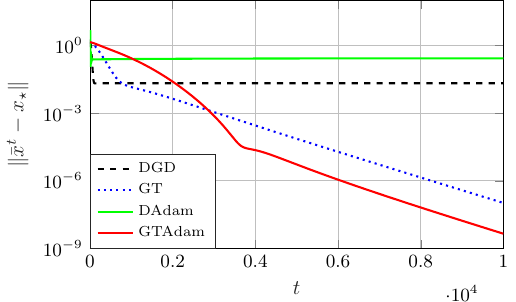}
		\caption{Distributed classification via logistic regression. Static set-up in which each of the $N=50$ agents is equipped with $m = 5$ points.}
		\label{fig:static}
	\end{figure}

\subsection{Distributed source localization in smart sensor networks}
\label{sec:localization}

The estimation of the exact position of a source is a key task in several
applications in multi-agent distributed estimation and learning. 
Here, we consider an online version of the static localization
problem considered in~\cite[Section~4.2]{rabbat2004distributed}. An acoustic
source is positioned at an unknown and time-varying location $\theta^t_{\text{target}} \in \R^2$. 
A network of $N$ sensors is capable to measure an isotropic signal related to 
such location and aims at cooperatively estimating $\theta^t_{\text{target}}$. Each sensor is
placed at a fixed location $c_i \in \R^2$ and takes, at each time
instant, a noisy measurement according to an isotropic propagation model
$\omega_{i}^{t} \triangleq \frac{A}{\|\theta^t_{\text{target}} - c_i\|^\gamma} + \epsilon_{i}^{t}$,
where $A > 0$, $\gamma \ge 1$ describes the attenuation
characteristics of the medium through which the signal propagates, and
$\epsilon_{i}^{t}$ is a zero-mean Gaussian noise with variance
$\sigma^2$. %
With this data, each node $i$ at each time $t \ge 0$ addresses a nonlinear least-squares online problem
\begin{equation*}
\begin{aligned}
	\min_{\x} \:\:
	\sum_{i=1}^N \Big( \omega_{i}^{t} - \frac{A}{\|\x - c_i\|^\gamma}\Big)^2.
\end{aligned}
\end{equation*}
We consider a network of $N=50$ agents randomly located according to a
two-dimensional Gaussian distribution with zero mean and variance
$a^2 I_2 = 100 I_2$. The agents want to track the location of a moving
target which starts at a random location $\theta^0_{\text{target}} \in \R^2$ generated according
to the same distribution of the agents.  The target moves along a
circle of radius $r = 0.5$ according to the following law
\begin{equation*}
\theta^t_{\text{target}} = \theta_{\text{center}} + r
\begin{bmatrix}
\cos(t/200)\\
\sin(t/200)
 \end{bmatrix},
 \end{equation*}
where $\theta_{\text{center}} \in \R^2$ represents the randomly generated circle center. 
We pick $\gamma = 1$, $A = 100$ and a noise variance $\sigma^2 = 0.001$. 
We take $\alpha = 0.05$ for \algo/, $\alpha = 0.02$ for \GT/, 
$\alpha = 0.05$ for \DGD/, and $\alpha = 0.0725$ for \DADAM/. The agents communicate according to a ring graph.
In Fig.~\ref{fig:target_cost} we compare the algorithm performance in terms of the (instantaneous) cost function evolution. 
\begin{figure}%
	\centering
	\includegraphics[scale=0.9]{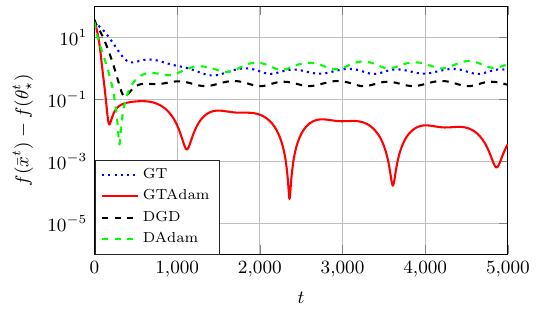}
	\caption{Distributed source localization. Cost function values obtained for a network of $N=50$ agents.}
	\label{fig:target_cost}
\end{figure}
Fig.~\ref{fig:target_regret} shows that the best performance in terms 
of average dynamic regret is obtained by \algo/.
\begin{figure}%
	\centering
	\includegraphics[scale=0.9]{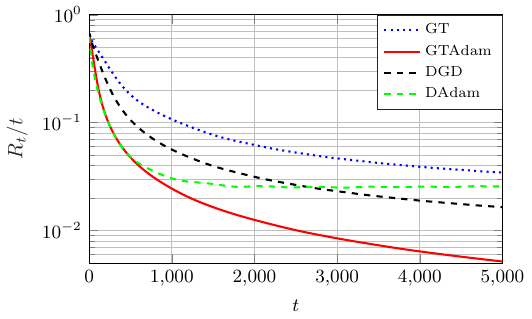}
	\caption{Distributed source localization. Average regret values obtained for a network of $N=50$ agents.}
	\label{fig:target_regret}
\end{figure}
\algo/ seems to achieve a smaller error with respect to the other algorithms. We make these comparisons by using $\theta^t_{\text{target}}$ as the optimal estimate associated to 
	the iteration $t$, but we note that the actual optimal solution may be slightly different since the noise $\epsilon_{i}^t$ affects the measurement of each agent.

\subsection{Distributed image classification via neural networks}
In this example, we consider an image classification problem in which $N$ nodes have to 
cooperatively learn how to correctly classify images.
We pick the Fashion-MNIST dataset~\cite{xiao2017fashion} consisting of 
black-and-white $28{\times}28$-pixels images of clothes belonging to $10$ different classes. Each agent $i$ has a local dataset $\mathcal{D}_i=\{(p_{i,k}, y_{i,k})\}_{k=1}^{m_i}$ 
consisting of $m_i$ images $p_{i,k}\in\R^{28\times 28}$ and their associated 
labels $y_{i,k}\in\until{10}$. 
The goal of the agents is to learn the parameters $\x_\star$ of a 
function $h(p;\x_\star)$ so that $h(p_{i,k};\x_\star)$ gives the correct label for $p_{i,k}$. 
The resulting optimization problem is%
\begin{align*}
\min_\x  \: \: 
\sum_{i=1}^N \frac{1}{m_i} \sum_{k=1}^{m_i} V(y_{i,k}, h(p_{i,k}, \x)) +C\|\x\|^2,
\end{align*}
where $V(\cdot)$ is the categorical cross-entropy loss, and $C > 0$ is a regularization parameter. The local cost function is
\begin{align*}
f_i(\x\mid\mathcal{D}_i)
\!\triangleq\! 
\E_{\mathcal{D}_i}[\ell_i(\x)]
\!=\!
\frac{1}{m_i} \sum_{k=1}^{m_i} \!V(y_{i,k}, h( p_{i,k}, \x))\! +\! \frac{C}{N} \|\x\|^2.
\end{align*}
We represent $h(\cdot)$ by a neural network 
with one hidden layer (with $300$ units with ReLU activation function) and an output layer with $10$ units. 
Moreover, we pick $N=16$ agents and associate each of them $m_i=3750$ labeled images for all $i$. 
We performed Monte Carlo simulations consisting of $100$ trials and each trial lasts $10$ epochs over the local datasets.
The results are reported In Fig.~\ref{fig:class_loss} and Fig.~\ref{fig:class_acc} in terms of the global training loss $f(\{\bar{\x}_{ep},\mathcal{D}_1, \dots, \mathcal{D}_N\}) \triangleq \sum_{i=1}^N f_i(\bar{\x}_{ep}\mid \mathcal{D}_i),$
with $\bar{\x}_{ep} \triangleq \tfrac{1}{N}\sum_{i=1}^N\x_{i,ep}$, %
and the average training accuracy 
$\psi({\{\bar{\x}_{ep}, \mathcal{D}_1, \dots, \mathcal{D}_N\}}) \triangleq \tfrac{1}{N}\sum_{i=1}^N \psi_i(\bar{\x}_{ep}\mid \mathcal{D}_i)$,
where $\psi_i(\bar{\x}_{ep}\mid \mathcal{D}_i)$ is the accuracy achieved with $\bar{\x}_{ep}$ on the local dataset of the agent $i$ at the end of epoch $ep$.
We take $\alpha = 0.001$ for \algo/, and $\alpha=0.1$ for \DGD/, \GT/, and \DADAM/. 
As it can be appreciated from Fig.~\ref{fig:class_loss} and Figure~\ref{fig:class_acc}, in both cases \algo/ outperforms the other algorithms.
\begin{figure}
	\centering
	\includegraphics[scale=0.9]{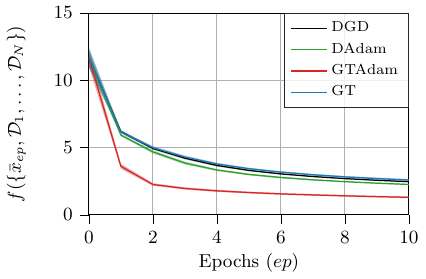}
	\caption{Distributed image classification. Mean and $3-$standard deviation band of the training loss.}
	\label{fig:class_loss}
\end{figure}
\begin{figure}
	\centering
	\includegraphics[scale=0.9]{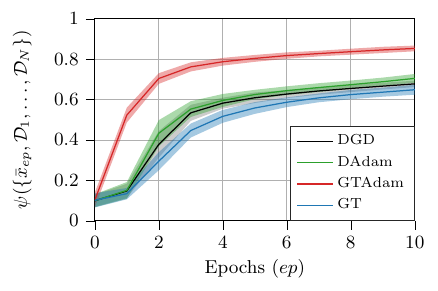}
	\caption{Distributed image classification. Mean and $3-$standard deviation band of the training accuracy.}
	\label{fig:class_acc}
\end{figure}

\section*{Conclusions}
\label{sec:conclusions}

We proposed \algo/, a novel distributed optimization
algorithm tailored for multi-agent online learning. Inspired by the
  popular Adam algorithm, our novel \algo/ is based
  on the gradient tracking distributed scheme which is enhanced with adaptive
  first- and second-order momentum estimates of the gradient. 
We provided theoretical bounds on the convergence of the proposed algorithm. 
Moreover, we tested \algo/ in three different scenarios showing a performance improvement with respect to state-of-the-art algorithms.

\bibliographystyle{ieeetr}

\begin{thebibliography}{10}

	\bibitem{li2022survey}
	X.~Li, L.~Xie, and N.~Li, ``A survey of decentralized online learning,'' {\em
		arXiv preprint arXiv:2205.00473}, 2022.
	
	\bibitem{simonetto2016class}
	A.~Simonetto, A.~Mokhtari, A.~Koppel, G.~Leus, and A.~Ribeiro, ``A class of
		prediction-correction methods for time-varying convex optimization,'' {\em
		IEEE Trans.~on Signal Processing}, vol.~64, no.~17, pp.~4576--4591, 2016.
	
	\bibitem{fazlyab2017prediction}
	M.~Fazlyab, S.~Paternain, V.~M. Preciado, and A.~Ribeiro,
		``Prediction-correction interior-point method for time-varying convex
		optimization,'' {\em IEEE Trans.~on Automatic Control}, vol.~63, no.~7,
		pp.~1973--1986, 2017.
	
	\bibitem{cavalcante2013distributed}
	R.~L. Cavalcante and S.~Stanczak, ``A distributed subgradient method for
		dynamic convex optimization problems under noisy information exchange,'' {\em
		IEEE Journal of Selected Topics in Signal Processing}, vol.~7, no.~2,
		pp.~243--256, 2013.
	
	\bibitem{towfic2014adaptive}
	Z.~J. Towfic and A.~H. Sayed, ``Adaptive penalty-based distributed stochastic
		convex optimization,'' {\em IEEE Transactions on Signal Processing}, vol.~62,
		no.~15, pp.~3924--3938, 2014.
	
	\bibitem{mateos2014distributed}
	D.~Mateos-N{\'u}nez and J.~Cort{\'e}s, ``Distributed online convex optimization
		over jointly connected digraphs,'' {\em IEEE Transactions on Network Science
		and Engineering}, vol.~1, no.~1, pp.~23--37, 2014.
	
	\bibitem{akbari2015distributed}
	M.~Akbari, B.~Gharesifard, and T.~Linder, ``Distributed online convex
		optimization on time-varying directed graphs,'' {\em IEEE Transactions on
		Control of Network Systems}, vol.~4, no.~3, pp.~417--428, 2015.
	
	\bibitem{hosseini2016online}
	S.~Hosseini, A.~Chapman, and M.~Mesbahi, ``Online distributed convex
		optimization on dynamic networks,'' {\em IEEE Transactions on Automatic
		Control}, vol.~61, no.~11, pp.~3545--3550, 2016.
	
	\bibitem{shahrampour2017distributed}
	S.~Shahrampour and A.~Jadbabaie, ``Distributed online optimization in dynamic
		environments using mirror descent,'' {\em IEEE Transactions on Automatic
		Control}, vol.~63, no.~3, pp.~714--725, 2017.
	
	\bibitem{akbari2019individual}
	M.~Akbari, B.~Gharesifard, and T.~Linder, ``Individual regret bounds for the
		distributed online alternating direction method of multipliers,'' {\em IEEE
		Transactions on Automatic Control}, vol.~64, no.~4, pp.~1746--1752, 2019.
	
	\bibitem{yi2020distributed}
	X.~Yi, X.~Li, L.~Xie, and K.~H. Johansson, ``Distributed online convex
		optimization with time-varying coupled inequality constraints,'' {\em IEEE
		Transactions on Signal Processing}, vol.~68, pp.~731--746, 2020.
	
	\bibitem{ram2010distributed}
	S.~S. Ram, A.~Nedi{\'c}, and V.~V. Veeravalli, ``Distributed stochastic
		subgradient projection algorithms for convex optimization,'' {\em Journal of
		optimization theory and applications}, vol.~147, no.~3, pp.~516--545, 2010.
	
	\bibitem{farina2019randomized}
	F.~Farina and G.~Notarstefano, ``Randomized block proximal methods for
		distributed stochastic big-data optimization,'' {\em IEEE Transactions on
		Automatic Control}, vol.~66, no.~9, pp.~4000--4014, 2021.
	
	\bibitem{pu2020distributed}
	S.~Pu and A.~Nedi{\'c}, ``Distributed stochastic gradient tracking methods,''
		{\em Mathematical Programming}, pp.~1--49, 2020.
	
	\bibitem{shi2015extra}
	W.~Shi, Q.~Ling, G.~Wu, and W.~Yin, ``Extra: An exact first-order algorithm for
		decentralized consensus optimization,'' {\em SIAM Journal on Optimization},
		vol.~25, no.~2, pp.~944--966, 2015.
	
	\bibitem{varagnolo2015newton}
	D.~Varagnolo, F.~Zanella, A.~Cenedese, G.~Pillonetto, and L.~Schenato,
		``{N}ewton-{R}aphson consensus for distributed convex optimization,'' {\em
		IEEE Transactions on Automatic Control}, vol.~61, no.~4, pp.~994--1009, 2015.
	
	\bibitem{dilorenzo2016next}
	P.~Di~Lorenzo and G.~Scutari, ``Next: In-network nonconvex optimization,'' {\em
		IEEE Transactions on Signal and Information Processing over Networks},
		vol.~2, no.~2, pp.~120--136, 2016.
	
	\bibitem{nedic2017achieving}
	A.~Nedi{\'c}, A.~Olshevsky, and W.~Shi, ``Achieving geometric convergence for
		distributed optimization over time-varying graphs,'' {\em SIAM Journal on
		Optimization}, vol.~27, no.~4, pp.~2597--2633, 2017.
	
	\bibitem{qu_harnessing_2018}
	G.~Qu and N.~Li, ``Harnessing {Smoothness} to {Accelerate} {Distributed}
		{Optimization},'' {\em IEEE Transactions on Control of Network Systems},
		vol.~5, no.~3, pp.~1245--1260, 2018.
	
	\bibitem{xu2017convergence}
	J.~Xu, S.~Zhu, Y.~C. Soh, and L.~Xie, ``Convergence of asynchronous distributed
		gradient methods over stochastic networks,'' {\em IEEE Transactions on
		Automatic Control}, vol.~63, no.~2, pp.~434--448, 2017.
	
	\bibitem{xin2018linear}
	R.~Xin and U.~A. Khan, ``A linear algorithm for optimization over directed
		graphs with geometric convergence,'' {\em IEEE Control Systems Letters},
		vol.~2, no.~3, pp.~315--320, 2018.
	
	\bibitem{scutari2019distributed}
	G.~Scutari and Y.~Sun, ``Distributed nonconvex constrained optimization over
		time-varying digraphs,'' {\em Mathematical Programming}, vol.~176, no.~1-2,
		pp.~497--544, 2019.
	
	\bibitem{zhu2010discrete}
	M.~Zhu and S.~Mart{\'\i}nez, ``Discrete-time dynamic average consensus,'' {\em
		Automatica}, vol.~46, no.~2, pp.~322--329, 2010.
	
	\bibitem{kia2019tutorial}
	S.~S. Kia, B.~Van~Scoy, J.~Cortes, R.~A. Freeman, K.~M. Lynch, and S.~Martinez,
		``Tutorial on dynamic average consensus: The problem, its applications, and
		the algorithms,'' {\em IEEE Control Systems Magazine}, vol.~39, no.~3,
		pp.~40--72, 2019.
	
	\bibitem{zhang2019distributed}
	Y.~Zhang, R.~J. Ravier, M.~M. Zavlanos, and V.~Tarokh, ``A distributed online
		convex optimization algorithm with improved dynamic regret,'' in {\em {IEEE}
		Conf.~on Decision and Control {(CDC)}}, pp.~2449--2454, 2019.
	
	\bibitem{notarnicola2020personalized}
	I.~Notarnicola, A.~Simonetto, F.~Farina, and G.~Notarstefano, ``Distributed
		personalized gradient tracking with convex parametric models,'' {\em IEEE
		Transactions on Automatic Control}, 2022.
	
	\bibitem{simonetto2019personalized}
	A.~Simonetto, E.~Dall'Anese, J.~Monteil, and A.~Bernstein, ``Personalized
		optimization with user's feedback,'' {\em Automatica}, vol.~131, p.~109767,
		2021.
	
	\bibitem{kingma2014adam}
	D.~P. Kingma and J.~Ba, ``Adam: A method for stochastic optimization,'' {\em
		arXiv preprint arXiv:1412.6980}, 2014.
	
	\bibitem{bock2018improvement}
	S.~Bock, J.~Goppold, and M.~Wei{\ss}, ``An improvement of the convergence proof
		of the adam-optimizer,'' {\em arXiv preprint arXiv:1804.10587}, 2018.
	
	\bibitem{chen2018convergence}
	X.~Chen, S.~Liu, R.~Sun, and M.~Hong, ``On the convergence of a class of
		adam-type algorithms for non-convex optimization,'' in {\em International
		Conference on Learning Representations}, 2018.
	
	\bibitem{reddi2018convergence}
	S.~J. Reddi, S.~Kale, and S.~Kumar, ``On the convergence of adam and beyond,''
		in {\em International Conference on Learning Representations}, 2018.
	
	\bibitem{zhou2018adashift}
	Z.~Zhou, Q.~Zhang, G.~Lu, H.~Wang, W.~Zhang, and Y.~Yu, ``Adashift:
		Decorrelation and convergence of adaptive learning rate methods,'' in {\em
		International Conference on Learning Representations}, 2018.
	
	\bibitem{nazari2019dadam}
	P.~Nazari, D.~A. Tarzanagh, and G.~Michailidis, ``Dadam: A consensus-based
		distributed adaptive gradient method for online optimization,'' {\em arXiv
		preprint arXiv:1901.09109}, 2019.
	
	\bibitem{mokhtari2016online}
	A.~Mokhtari, S.~Shahrampour, A.~Jadbabaie, and A.~Ribeiro, ``Online
		optimization in dynamic environments: Improved regret rates for strongly
		convex problems,'' in {\em {IEEE} Conference on Decision and Control (CDC)},
		pp.~7195--7201, 2016.
	
	\bibitem{dall2020optimization}
	E.~Dall'Anese, A.~Simonetto, S.~Becker, and L.~Madden, ``Optimization and
		learning with information streams: Time-varying algorithms and
		applications,'' {\em IEEE Signal Processing Magazine}, vol.~37, no.~3,
		pp.~71--83, 2020.
	
	\bibitem{li2020online}
	Y.~Li, G.~Qu, and N.~Li, ``Online optimization with predictions and switching
		costs: Fast algorithms and the fundamental limit,'' {\em IEEE Transactions on
		Automatic Control}, vol.~66, no.~10, pp.~4761--4768, 2020.
	
	\bibitem{horn2012matrix}
	R.~A. Horn and C.~R. Johnson, {\em Matrix analysis}.
	\newblock Cambridge university press, 2012.
	
	\bibitem{nedic2009distributed}
	A.~Nedi{\'c} and A.~Ozdaglar, ``Distributed subgradient methods for multi-agent
		optimization,'' {\em IEEE Transactions on Automatic Control}, vol.~54, no.~1,
		pp.~48--61, 2009.
	
	\bibitem{rabbat2004distributed}
	M.~Rabbat and R.~Nowak, ``Distributed optimization in sensor networks,'' in
		{\em International Symposium on Information Processing in Sensor Networks},
		pp.~20--27, 2004.
	
	\bibitem{xiao2017fashion}
	H.~Xiao, K.~Rasul, and R.~Vollgraf, ``Fashion-mnist: a novel image dataset for
		benchmarking machine learning algorithms,'' {\em arXiv preprint
		arXiv:1708.07747}, 2017.
	
	\end{thebibliography}

\appendix

\section{Appendix}
\label{appendix}

We report a lemma that will be used in the proof of Lemma~\ref{lemma:xbar} (cf. Appendix~\ref{sec:proof_lemma_7}).

\begin{lemma}\label{lemma:strongandLipconvergence}
	Let $f(x):\R^n\to\R$ be $\sigma$-strongly convex and with $L$-Lipschitz continuous gradient. Moreover, let $D\in\R^{n\times n}$ be positive definite diagonal matrix such that $D_{ii}\in[\epsilon, M]$ for all $i=1,\dots,n$ with $M\geq\epsilon>0$ and $M<\infty$. Let $\bL=ML$ and $\bs=\epsilon\sigma$. Let $\xt = \xtm - \alpha D\nabla f(\xtm)$, with $\alpha\in(0,\frac{2}{\bL}]$. Then%
	$\|\xt-\x_\star\| \leq \max\{(1-\alpha\bs),(1-\alpha\bL) \}\|\xtm - \x_\star\|$.
\end{lemma}
\begin{proof}
	Let $h(x)$ be a function such that $\nabla h(x)= D \nabla f(x)$ for all $x$. It can be easily shown that $h$ has $\bL$-Lipschitz continuous gradients, in fact 
	\begin{align*}
	&\|\nabla h(x) - \nabla h(y)\| = \|D\nabla f(x) - D\nabla f(y)\|\\
	&\leq \|D\|\|\nabla f(x) - \nabla f(y)\|\leq \|D\|L \|x-y\|\leq ML \|x-y\|.
	\end{align*}
	Moreover $h$ is $\bs$-strongly convex, since $\nabla^2h(x)=D\nabla^2 f(x)\succeq D\sigma I\geq \epsilon\sigma I$. Define $g(x)=h(x)-\frac{\bs}{2}\|x\|^2$. Notice that, by definition, $g$ is convex and with $(\bL-\bs)$-Lipschitz continuous gradient. Thus, by definition we have
	\begin{equation}\label{1}
	\langle \nabla g(x) - \nabla g(y), x-y\rangle \geq \tfrac{1}{\bL-\bs}\|\nabla g(x) - \nabla g(y)\|^2.
	\end{equation}
	Now, by using the definition of $g$ one has%
	\begin{align}
	&\langle \nabla h(x) - \bs x- \nabla h(y) + \bs y, x-y\rangle 
	\notag\\
	&= \langle \nabla h(x) - \nabla h(y), x-y\rangle - \bs \|x-y\|^2.
	\label{2}
	\end{align}%
	Moreover
	\begin{align}
	\|\nabla g(x) - \nabla g(y)\|^2 &= \|\nabla h(x) - \bs x- \nabla h(y) + \bs y\|^2\nonumber\\
	&= \|\nabla h(x) - \nabla h(y)\|^2 + \bs^2\| x- y\|^2
	\notag\\
	&\hspace{3ex} - 2\bs\langle \nabla h(x) - \nabla h(y), x-y\rangle.
	\label{3}
	\end{align}
	By combining~\eqref{1},~\eqref{2}, and~\eqref{3} we get
	\begin{align}
	&\langle \nabla h(x) - \nabla h(y), x-y\rangle
	\notag\\
	&\qquad 
	\geq\tfrac{\bs\bL}{\bs+\bL} \|x-y\|^2 + \tfrac{1}{\bs+\bL}\|\nabla h(x) - \nabla h(y)\|^2.\label{eq:lemma_intermediate_result}
	\end{align}
	Now, by using the update rule, one has
	\begin{align*}
	&\|\xt-\x_\star\|^2 = \|\xtm - \alpha D\nabla f(\xtm)-\x_\star\|^2
	\\
	&=\|\xtm -\x_\star\|^2 -2\alpha\langle D\nabla f(\xtm),\xtm-\x_\star\rangle +\alpha^2 \|D\nabla f(\xtm)\|^2
	\\
	&=\|\xtm -\x_\star\|^2 -2\alpha\langle D\nabla f(\xtm) - D\nabla f(\x_\star),\xtm-\x_\star\rangle \\
	&\hspace{3ex}+\alpha^2 \|D\nabla f(\xtm) - D\nabla f(\x_\star)\|^2.
	\end{align*}
	By using the result~\eqref{eq:lemma_intermediate_result} with $\nabla h(x) = D\nabla f(x)$, we have
	\begin{align*}
	&\|\xt-\x_\star\|^2 \leq\|\xtm -\x_\star\|^2 +\alpha^2 \|D\nabla f(\xtm) - D\nabla f(\x_\star)\|^2\\
	&\hspace{1ex}-2\alpha\tfrac{\bs\bL}{\bs+\bL} \|\xtm-\x_\star\|^2 - \tfrac{2\alpha}{\bs+\bL}\|D\nabla f(\xtm) - D\nabla f(\x_\star)\|^2
	\\
	&=\left(1-2\alpha\tfrac{\bs\bL}{\bs+\bL}\right)\|\xtm-\x_\star\|^2
	\\
	&\hspace{3ex}+\alpha\left(\alpha - \tfrac{2}{\bs+\bL}\right)\|D\nabla f(\xtm) - D\nabla f(\x_\star)\|^2
	\\
	&\leq\left(1-2\alpha\tfrac{\bs\bL}{\bs+\bL}\right)\|\xtm-\x_\star\|^2+\alpha\left(\alpha \bL^2 -\tfrac{2\bs^2}{\bs+\bL}\right)\|\xtm - \x_\star\|^2 
	\\
	&\leq \max\{(1-\alpha\bs)^2,(1-\alpha\bL)^2 \}\|\xtm - \x_\star\|^2. 
	\end{align*}
	The proof follows by taking the square root of both sides.
\end{proof}

\subsection{Proof of Lemma~\ref{lemma:m_bound}}\label{sec:proof_lemma_1}

	By using the update~\eqref{eq:m_avg}, we can write
	\begin{align}
	\!\!\!\norm{\avgmt} \!=\! \norm{\beta_1\avgmtm \!+\! (1\!-\!\beta_1)\avgstm}
	\!\leq\! \beta_1\!\norm{\avgmtm}\! +\! (1\!-\!\beta_1)\!\norm{\avgstm}
	\label{eq:bound_avgtm2}
	\end{align}
	in which we use the triangle inequality. Regarding the term
	$\norm{\avgstm}$, we use the relation
	$\avgstm = \tfrac{1}{N}\sum_{i=1}^N\nabla \fitm(\xitm)$, and we add
	$\tfrac{1}{N}\sum_{i=1}^N\nabla \fitm(\xstartm) = 0$, thus obtaining
	\begin{align}
	&\!\|\avgstm\|\!=\! \norm{\tfrac{1}{N}\!\sum_{i=1}^N\!\nabla \fitm(\xitm)\!-\!\tfrac{1}{N}\!\sum_{i=1}^N\!\nabla\fitm(\xstartm)}
	\!\stackrel{(a)}\leq\! \tfrac{L}{N}\!\sum_{i=1}^N\!\norm{\xitm\!-\! \xstartm}
	\notag
	\\ &
	\!\stackrel{(b)}
	\leq\! \tfrac{L}{\srN}\norm{\xtm\!-\!\one\xstartm}
	\!\stackrel{(c)}\leq\! \tfrac{L}{\srN}\norm{\xtm\!-\!\one\avgxtm}\! +\! L\norm{\avgxtm\!-\!\xstartm}\label{eq:bound_avgstm},
	\end{align}
	where in $(a)$ we exploit the Lipschitz continuity of the gradients of the cost functions (cf. Assumptions~\ref{assumption:lipschitz}), in $(b)$ we use
	the basic algebraic property
	$\sum_{i=1}^{N}\|\theta_i\| \leq \sqrt{N}\|\theta\|$ for a generic
	vector $\theta \triangleq \col(\theta_1, \dots, \theta_N)$, and in $(c)$ we add and subtract 
	the term $\one \avgxtm$ and apply the triangle inequality. The proof follows by combining the bounds~\eqref{eq:bound_avgtm2} and~\eqref{eq:bound_avgstm}.

\subsection{Proof of Lemma~\ref{lemma:d}}\label{sec:proof_lemma_3}

	By using~\eqref{eq:d_global} and~\eqref{eq:d_avg}, one has
		\begin{align}
		& \norm{\dt\!-\!\one\avgdt}\!=\!\norm{\left(I-\ones\right)(\Vt + \epsilon I)^{-1/2}\mt}
		\notag\\
		&
		\hspace{3ex}
		\stackrel{(a)}{\leq}\!\norm{(\Vt \!+\! \epsilon I)^{-1/2}} \! \norm{\mt}
		\stackrel{(b)}{\leq}\!\tfrac{1}{\sqrt{\epsilon}}\!\norm{\mt}
		\notag\\
		&
		\hspace{3ex}
		\stackrel{(c)}{\leq}\!\tfrac{1}{\sqrt{\epsilon}}\!\norm{\mt \!-\!\one\avgmt\!} 
		\!+\! \tfrac{\sqrt{N}}{\sqrt{\epsilon}}\!\norm{\avgmt\!}\!,\label{eq:bound_d}
		\end{align}
		where in $(a)$ we apply the Cauchy-Schwarz inequality combined with $\norm{I-\ones} \leq 1$, in $(b)$ we use the bound $\norm{(\Vt + \epsilon I)^{-1/2}} \leq \frac{1}{\sqrt{\epsilon}}$ (justified by the fact that $\vt \ge 0$ for all $t \ge 0$), in $(c)$ we add and subtract within the norm $\one\avgmt$ and apply the triangle inequality and an algebraic property.
		The proof follows by using Lemma~\ref{lemma:m_bound} and~\ref{lemma:m} in~\eqref{eq:bound_d}.

\subsection{Proof of Lemma~\ref{lemma:s}}\label{sec:proof_lemma_4}

	By combining~\eqref{eq:s_global} and~\eqref{eq:s_avg} one has
		\begin{align}
		&\|\st - \one \avgst \| 
		\notag\\
		&
		= \left\|W\stm  + \gt - \gtm - \one\!\Bigl(\!\avgstm  + \tfrac{1}{N}\sum_{i=1}^N (\git - \gitm)\! \Bigr)\!\right\|\nonumber\\
		&\stackrel{(a)}\leq\left\|\left(W - \ones\right)(\stm -\one\avgstm)\right\| \!+\! \left\|\left( I-\ones \right)(\gt - \gtm)\right\|
		\nonumber\\
		&\stackrel{(b)}=\sigma_W\norm{\stm -\one\avgstm} + \norm{\gt - \gtm},\label{eq:gt_gtm}
		\end{align}
		where $(a)$ uses $\one \in \ker \left (W - \ones \right )$ and the triangle inequality, and $(b)$ combines the Cauchy-Schwarz inequality with the bounds $\norm{W - \ones} \leq \sigma_W$ and $\norm{I-\ones} \leq 1$.
		Let $\tilde{\mathbf{g}}^t \triangleq \col(\nabla f_{1}^{t+1}(x_{1}^{t}),\dots,\nabla f_{N}^{t+1}(x_{N}^{t}))$ and manipulate the term $\norm{\gt - \gtm}$ in~\eqref{eq:gt_gtm} as
	\begin{align}\nonumber
	&\|\gt - \gtm\| \leq \|\gt - \tilde{\mathbf{g}}^t\| + \norm{\tilde{\mathbf{g}}^t - \gtm}
	\\
	&
	\stackrel{(a)}\leq L\|\xt - \xtm\| + \|\tilde{\mathbf{g}}^t - \gtm\| \stackrel{(b)}\leq L\norm{\xt - \xtm} + \sqrt{N}\eta^t
	\notag\\
	&\stackrel{(c)}= L\|W\xtm -\alpha\dt - \xtm\| + \sqrt{N}\eta^t,\label{eq:gradients_difference}
	\end{align}
	where in $(a)$ we use the Lipschitz continuity of the gradients of the cost functions (cf. Assumption~\ref{assumption:lipschitz}), $(b)$ uses the variable $\eta^t$ (cf~\eqref{eq:bound_optimum_gradient}), and $(c)$ uses the
	update~\eqref{eq:x_global} of $\xt$. Let us manipulate the first term on the
	right-hand side of~\eqref{eq:gradients_difference}:
	\begin{align}
	& \norm{W\xtm -\alpha\dt - \xtm}
	\stackrel{(a)}{=}
	\! 
	\norm{(W-I)(\xtm-\one\avgxtm) -\alpha\dt}
	\nonumber\\
	& \stackrel{(b)}{\leq} 
	\! 
	\|W\!-\!I\|\|\xtm\!-\!\one\avgxtm\| \!+\! \alpha\|\dt\!-\!\one\avgdt\| 
	\!+\!\alpha\|\one\avgdt\|,\label{eq:passage}
	\end{align}
	where $(a)$ uses the fact that $\ker \left(W-I\right) = \textrm{span}(\one )$ and in $(b)$ we add and subtract 
	the term $\one\avgdt$ within the norm and we apply the triangle inequality and the Cauchy-Schwarz inequality. Regarding $\|\one\avgdt\|$, we use~\eqref{eq:d_global} and~\eqref{eq:d_avg} 
	to write
		\begin{align}
		&\|\one\avgdt\| = \left\|\ones \dt\right\|=\left\|\ones (\Vt + \epsilon I)^{-1/2} \mt\right\|
		\nonumber\\&
		\stackrel{(a)}{\leq}\!\tfrac{1}{\sqrt{\epsilon}}\!\norm{\mt}\!
		\stackrel{(b)}{\leq}\! \tfrac{1}{\sqrt{\epsilon}}\!\norm{\mt \!- \!\one\avgmt} + \tfrac{\srN}{\sqrt{\epsilon}}\!\norm{\avgmt}\!,
		\label{eq:avgdt_bound}
		\end{align}
		where in $(a)$ we apply the Cauchy-Schwarz inequality and the bounds $\norm{\ones} \leq 1$ and $\norm{(\Vt + \epsilon)^{-1/2}} \leq \tfrac{1}{\sqrt{\epsilon}}$, in $(b)$ we add and subtract within the norm the term $\one\avgmt$, apply the triangle inequality, and use an algebraic property. By combining~\eqref{eq:passage} and~\eqref{eq:avgdt_bound}, we bound~\eqref{eq:gradients_difference} as
		\begin{align}
			&\|\gt - \gtm\| \leq L\|W-I\|\|\xtm-\one\avgxtm\| +\alpha L\|\dt \! -\! \one\avgdt\| 
			\notag\\
			&\hspace{1ex}+
		\alpha\tfrac{ L}{\sqrt{\epsilon}}\norm{\mt - \one\avgmt}\! +\! \alpha\tfrac{ L\srN}{\sqrt{\epsilon}}\norm{\avgmt}\!+\! \srN\eta^t.\label{eq:gradients_difference_bound}
		\end{align}
		Now, by using the bound~\eqref{eq:gradients_difference_bound} within~\eqref{eq:passage}, we get
		\begin{align}
			&\norm{\st - \one\avgst}\leq \sigma_W\norm{\stm - \one\avgstm}
			\notag\\
			&\hspace{1ex}+L\|W-I\|\|\xtm-\one\avgxtm\|+\alpha L\|\dt-\one\avgdt\| 
		\notag\\
		&\hspace{1ex}+
		\alpha\tfrac{ L}{\sqrt{\epsilon}}\norm{\mt - \one\avgmt}\!+\! \alpha\tfrac{ L\srN}{\sqrt{\epsilon}}\norm{\avgmt}\!+\!\srN\eta^t. \label{eq:final_passage}
		\end{align}
		The proof follows by using Lemma~\ref{lemma:m_bound},~\ref{lemma:m} and~\ref{lemma:d} to bound $\norm{\avgmt}$, $\norm{\mt-\one\avgmt}$, and $\norm{\dt - \one\avgdt}$.

\subsection{Proof of Lemma~\ref{lemma:x}}\label{sec:proof_lemma_5}

	By combining~\eqref{eq:x_global} and~\eqref{eq:x_avg}, we have
	\begin{align*}
	\|\xt - \one \avgxt\| &=\| W\xtm - \alpha \dt - \one\avgxtm + \alpha \one \avgdt\|
	\\
    &\stackrel{(a)}\leq \|W\xtm - \one\avgxtm\| + \alpha \|\dt -\one \avgdt\|
    \\
	&\stackrel{(b)}\leq \sigma_W\|\xtm - \one\avgxtm\| + \alpha \|\dt -\one \avgdt\|,
	\end{align*}
	where in $(a)$ we apply the triangle inequality and $(b)$ follows by $\norm{W - \ones} \leq \sigma_W$. 
	The proof follows by Lemma~\ref{lemma:d}.

\subsection{Proof of Lemma~\ref{lemma:z}}\label{sec:proof_lemma_6}

	From the updates of $\avgst$ and $\avgmt$ (cf.~\eqref{eq:s_avg},~\eqref{eq:m_avg}), we get
	\begin{align*}
	& \|\avgst\! -\! \avgmt\|\! = \! \bigg\| \avgstm\! +\!\nablaft-\nablaftm 
	\\
	&\hspace{3ex}- \beta_1\avgmtm - (1-\beta_1)\avgstm\bigg\|
	\\
	& \stackrel{(a)} 
	\leq \beta_1 \| \avgstm-\avgmtm \| \!+\! \left\| \nablaft-\nablaftm \right\|,
	\end{align*}
	where $(a)$ uses the triangle inequality. 
	By adding and subtracting within the second norm $\tfrac{1}{N}\sum_{i=1}^{N}\nabla \fit(\avgxt)$ and $\tfrac{1}{N}\sum_{i=1}^{N}\nabla \fit(\xitm)$, we use the triangle inequality to obtain
	\begin{align}
	&\|\avgst - \avgmt\| 
	\leq\beta_1\|\avgstm - \avgmtm\|
	\notag\\
	&\hspace{3ex}+ \left\|\nablaft-\tfrac{1}{N}\sum_{i=1}^{N}\nabla\fit(\avgxt)\right\|
	\nonumber\\
	&\hspace{3ex}+ \left\|\tfrac{1}{N}\sum_{i=1}^{N}\nabla \fit(\xitm)-\nablaftm\right\|\nonumber\\
	&\hspace{3ex}+\left\|\tfrac{1}{N}\sum_{i=1}^{N}\nabla \fit(\avgxt)-\tfrac{1}{N}\sum_{i=1}^{N}\nabla \fitm(\xitm)\right\|\nonumber\\
	&
	\stackrel{(a)}\leq\beta_1\|\avgstm - \avgmtm\|+\tfrac{L}{\srN}\|\xt-\one\avgxt\|+\tfrac{1}{\srN}\eta^{t}
	\notag\\
	&\hspace{3ex}+\tfrac{L}{\srN}\|\xtm-\one\avgxt\|\label{eq:z_passage1},
	\end{align}
	where in $(a)$ we use the Lipschitz continuity of the gradients of the cost functions (cf. Assumptions~\ref{assumption:lipschitz}) for the second and the third norm, and we use $\eta^t$ (cf.~\eqref{eq:bound_optimum_gradient}).
	Now, we replace
	$\avgxt$ with its update~\eqref{eq:x_avg} within the last term of~\eqref{eq:z_passage1} obtaining
		\begin{align}
		&\|\avgst - \avgmt\| \leq
		\beta_1\|\avgstm - \avgmtm\|+\tfrac{L}{\srN}\|\xt-\one\avgxt\|
		\notag\\
		&\hspace{3ex}+\tfrac{L}{\srN}\eta^{t} +\tfrac{L}{\srN}\|\one\avgxtm-\alpha\one\avgdt-\xtm\|
		\notag\\
		&\stackrel{(a)}\leq
		\beta_1\|\avgstm - \avgmtm\|+\tfrac{L}{\srN}\|\xt-\one\avgxt\|+\tfrac{L}{\srN}\eta^{t}
		\notag\\
		&\hspace{3ex} 
		+\tfrac{L}{\srN}\|\xtm - \one\avgxtm\| + \alpha\tfrac{ L}{\sqrt{\epsilon}\srN}\norm{\mt - \one\avgmt}
		\notag\\
		&\hspace{3ex}+ \alpha\tfrac{ L}{\sqrt{\epsilon}}\norm{\avgmt},
		\end{align}
		where in $(a)$ we use \eqref{eq:avgdt_bound} to bound $\norm{\one\avgdt}$. The proof follows by using Lemma~\ref{lemma:x},~\ref{lemma:m_bound}, and~\ref{lemma:m} to bound $\norm{\xt-\one\avgxt}$, $\norm{\avgmt}$, and $\norm{\mt - \one\avgmt}$, respectively.

\subsection{Proof of Lemma~\ref{lemma:xbar}}\label{sec:proof_lemma_7}

	By using~\eqref{eq:x_avg}, one has
	\begin{align}\nonumber
	\norm{\avgxt - \xstart} &= \norm{\avgxtm - \alpha \avgdt - \xstart}
	\\
	&
	\stackrel{(a)}\leq \norm{\avgxtm - \alpha \avgdt - \xstartm}+ \zeta^t,\nonumber
	\end{align}
	where in $(a)$ we add and subtract within the norm $\xstartm$, use the triangle inequality, 
	and use $\zeta^t$ (cf.~\eqref{eq:bound_optimum_gradient}). 
	Now, we add and subtract within the norm
	$\alpha \frac{\one^\top(\Vt+\epsilon I)^{-1/2}\one}{N^2}\nabla
	\ftm(\avgxtm)$ and we use the triangle inequality to write
	\begin{align}\nonumber
	&\|\avgxt - \xstart\| \leq \left\|\avgxtm - \alpha \tfrac{\one^\top(\Vt+\epsilon I)^{-1/2}\one}{N^2}\nabla \ftm(\avgxtm) - \xstartm\right\|
	\nonumber\\
	&\hspace{3ex}+ \alpha\left\| \tfrac{\one^\top(\Vt+\epsilon I)^{-1/2}\one}{N^2}\nabla \ftm(\avgxtm)- \avgdt\right\|\! +\! \zeta^t.\label{eq:inequality}
	\end{align}
	Consider the second term of \eqref{eq:inequality} and use \eqref{eq:d_avg} to write
		\begin{align}
		&\alpha\left\| \tfrac{\one^\top(\Vt+\epsilon I)^{-1/2}\one}{N} \tfrac{\nabla \ftm(\avgxtm)}{N}-  \avgdt\right\|\nonumber
		\\
		&= \alpha \bigg\| \tfrac{\one^\top(\Vt+\epsilon I)^{-1/2}\one}{N}\tfrac{\nabla \ftm(\avgxtm)}{N}-  \tfrac{\one^\top(\Vt+\epsilon I)^{-1/2}}{N}\mt\bigg\|
		\notag
		\\
		&\stackrel{(a)}\leq\alpha\norm{\tfrac{ \one^\top(\Vt + \epsilon I)^{-1/2}\one}{N}\left(\tfrac{\nabla \ftm(\avgxtm)}{N}-\avgmt\right)}
		\notag\\
		&\hspace{3ex} + \alpha\norm{\tfrac{\one^\top(\Vt+\epsilon I)^{-1/2}}{N}(\mt - \one\avgmt)}
		\notag\\
		&\stackrel{(b)}
		\leq
		\tfrac{\alpha}{\sqrt{\epsilon}}\norm{\tfrac{\nabla \ftm(\avgxtm)}{N}\!-\!\avgmt} 
		\!+\! \tfrac{\alpha}{\sqrt{\epsilon}\srN}\norm{\mt\!-\!\one\avgmt},\label{eq:last}
		\end{align}
		where in $(a)$ we add and subtract within the norm the term $\frac{\one^\top(\Vt+\epsilon I)^{-1/2}\one}{N}\avgmt$ and we apply the triangle inequality, in $(b)$ we apply the Cauchy-Schwarz inequality combined with the bounds $\norm{\frac{\one^\top (\Vt + \epsilon I)^{-1/2}\one}{N}} \leq \frac{1}{\sqrt{\epsilon}}$ and $\norm{\frac{\one^\top (\Vt + \epsilon I)^{-1/2}}{N}} \leq \frac{1}{\sqrt{\epsilon}\srN}$.
		Now, we add and subtract the term $\frac{1}{N}\sum_{i=1}^{N}\nabla \fitm(\xitm)$ and then we use the triangle inequality to rewrite the first term of the second member of~\eqref{eq:last} as
	\begin{align}
	&\alpha\tfrac{1}{\sqrt{\epsilon}}\|\tfrac{\nabla \ftm(\avgxtm)}{N}-\avgmt\|\leq\alpha\tfrac{1}{\sqrt{\epsilon}} 
	\Big\|\tfrac{1}{N}\sum_{i=1}^{N}\nabla \fitm(\xitm)-\avgmt\Big\|
	\nonumber\\
	&\hspace{3ex}+ \alpha\tfrac{1}{\sqrt{\epsilon}}
	\Big\|\tfrac{\nabla \ftm(\avgxtm)}{N}
	- \tfrac{1}{N}\sum_{i=1}^{N}\nabla \fitm(\xitm)\Big\|
	\nonumber\\
	&\stackrel{(a)}=\alpha\tfrac{1}{\sqrt{\epsilon}}
	\Big\|\tfrac{1}{N}\sum_{i=1}^{N}\nabla \fitm(\xitm)-\beta_1\avgmtm-(1-\beta_1)\avgstm\Big\|
	\nonumber\\
	&\hspace{3ex} + \alpha\tfrac{1}{\sqrt{\epsilon}}
	\Big\|\tfrac{\nabla \ftm(\avgxtm)}{N} - \tfrac{1}{N}\sum_{i=1}^{N}\nabla \fitm(\xitm)\Big\|
	\nonumber\\
	&\stackrel{(b)}\leq \alpha\tfrac{\beta_1}{\sqrt{\epsilon}}\|\avgstm - \avgmtm\| 
	+ \alpha\tfrac{L}{\sqrt{\epsilon}\srN}\|\xtm-\one\avgxtm\|,
	\label{eq:other_term}
	\end{align}
	where in $(a)$ we use \eqref{eq:m_avg}, $(b)$ uses the relation $\avgstm = \nablaftm$, and %
	the Lipschitz continuity of the gradients of the cost functions (cf. Assumption~\ref{assumption:lipschitz}). %
	Next, in order to bound the right-hand side of  \eqref{eq:inequality}, first notice that
	$\tfrac{1}{\sqrt{G+\epsilon}} < \tfrac{\one^\top(\Vt+\epsilon I)^{-1/2}\one}{N} < \tfrac{1}{\sqrt{\epsilon}}$.
	Moreover, being $\ftm$ $\mu$-strongly convex for all $t\ge0$ (cf. Assumption~\ref{assumption:strong}) and having $L$-Lipschitz continuous gradients (cf. Assumption~\ref{assumption:lipschitz}), we apply Lemma~\ref{lemma:strongandLipconvergence} (in the Appendix) to write
	\begin{align}
	\!\norm{\avgxtm \!-\! \alpha \tfrac{\one^\top(\Vt+\epsilon I)^{-1/2}\one}{N^2}\nabla \ftm(\avgxtm) \!-\! \xstartm} 
	\!\leq\! \phi\norm{\avgxtm\!-\! \xstartm},\label{eq:phi}
	\end{align}
	where
	$\phi \triangleq \max\left\{\left|1-\tfrac{\alpha}{\sqrt{\epsilon+G}}\mu\right|,\left|1-\tfrac{\alpha}{\sqrt{\epsilon}}L\right|\right\}$.
	If we take $\alpha < \min\left\{\tfrac{\sqrt{\epsilon+G}}{\mu}, \tfrac{\sqrt{\epsilon}}{L}\right\}$,	then it holds $\phi = 1 - \alpha\delta$, where $\delta$ is defined in the statement of Theorem~\ref{th:convergence}. By combining the latter with~\eqref{eq:other_term} and~\eqref{eq:phi}, it is possible to upper bound \eqref{eq:inequality} as
	\begin{align}
	\notag
	& \|\avgxt - \xstart\|\leq (1-\alpha\delta)\|\avgxtm-\xstartm\|+\alpha\tfrac{\beta_1}{\sqrt{\epsilon}}\|\avgstm - \avgmtm\|
	\notag
	\\
	& \hspace{1ex}+ \tfrac{\alpha}{\sqrt{\epsilon}\srN}\norm{\mt\!-\!\one\avgmt} 
	+ \tfrac{\alpha L}{\sqrt{\epsilon}\srN}\norm{\xtm \!-\! \one\avgxtm} + \zeta^t.
	\label{eq:solution_last}
	\end{align}
	The proof follows by invoking Lemma~\ref{lemma:m} to bound $\norm{\mt - \avgmt}$ within~\eqref{eq:solution_last}.

	\subsection{Proof of Corollary~\ref{cor:agent_regret}}
	\label{sec:proof_cor_agent_regret}
	
	We add and subtract $\ftm(\avgxtm)$ to $\ftm(\xitm) - \ftm(\xstartm)$, obtaining
	\begin{align}
	&\ftm(\xitm) - \ftm(\xstartm) 
	 = \ftm(\xitm) - \ftm(\avgxtm) + \ftm(\avgxtm) - \ftm(\xstartm) \notag
	 \\
	 &
	\stackrel{(a)}\leq \ftm(\xitm) - \ftm(\avgxtm)  + \frac{L}{2}\norm{\avgxtm - \xstartm}^2
	\notag\\
	&
	\hspace{1ex}
	\stackrel{(b)}\leq \nabla \ftm(\avgxtm)^\top(\xitm\! -\! \avgxtm) \!+ \!\frac{L}{2}\norm{\xitm\! -\! \avgxtm}^2\! +\! \frac{L}{2}\norm{\avgxtm \!-\! \xstartm}^2\!\!,
	\label{eq:difference_f_second_result}
	\end{align}
	where in $(a)$ we apply~\eqref{eq:f_minus_fstar} and in $(b)$ we use the Lipschitz continuity of the gradients of the cost functions (cf. Assumption~\ref{assumption:lipschitz}). Being $\nabla\ftm(\xstartm) = 0$, we rewrite~\eqref{eq:difference_f_second_result} as
	\begin{align}
	&\ftm(\xitm) - \ftm(\xstartm) 
	\leq (\nabla \ftm(\avgxtm) - \nabla \ftm(\xstartm))^\top(\xitm - \avgxtm) 
	\notag\\
	&\hspace{3ex}
	+ \tfrac{L}{2}\norm{\xitm - \avgxtm}^2 
	+ \tfrac{L}{2}\norm{\avgxtm - \xstartm}^2
	\notag\\
&\!\!\!
\stackrel{(a)}{\leq}\!
L\norm{\avgxtm\!-\!\xstartm}\!\norm{\xitm\!-\! \avgxtm}
	\!+\!\tfrac{L}{2}\!\norm{\xitm\! -\! \avgxtm}^2
	\!+\!\tfrac{L}{2}\!\norm{\avgxtm\!-\! \xstartm}^2\!,\label{eq:regret_j_intermediate}
\end{align}
where in $(a)$ we use the Cauchy-Schwarz inequality and the Lipschitz continuity of the gradients of the cost functions (cf. Assumption~\ref{assumption:lipschitz}). Now, we notice that both $\norm{\avgxtm -  \xstartm}$ and $\norm{\xitm - \avgxtm}$ represent a component of the vector $y^t$ defined in~\eqref{eq:yt}, and thus, can be both upper bounded by $\norm{y^t}$. 
	Hence, the inequality~\eqref{eq:regret_j_intermediate} can be elaborated as
	\begin{align}\label{eq:diference_cost_two}
	\ftm(\xitm) - \ftm(\xstartm)  
	& \leq 2L\norm{y^t}^2.
	\end{align}
	By summing over $t$ the inequality in~\eqref{eq:diference_cost_two}, we bound $R_{T,i}$ as
	\begin{align}
	R_{T,i}
	&\leq 2L\sum_{t=1}^T\norm{y^t}^2
	\stackrel{(a)}{\leq} 2L\lambda_1^2\sum_{t=1}^T\norm{y^t}^2_\ped,
	\label{eq:regret_j_final}
	\end{align}
	where in $(a)$ we apply~\eqref{eq:lambda_1}. As done above to prove~\eqref{eq:result_1}, the proof follows by combining~\eqref{eq:regret_j_final},~\eqref{eq:bound_y}, and~\eqref{eq:lambda_2}.

	\subsection{Proof of Corollary~\ref{cor:static}}
	\label{sec:proof_cor_static}
	Using the same arguments of Theorem~\ref{th:convergence} we start from~\eqref{eq:bound_y}.
	Differently from the dynamic case, in the static set-up we have $\nabla\fitm(x) = \nabla f_i(x)$ for all $t$ and $i$, leading to $\xstartm = x_\star$ for all $t$. Thus, we can combine~\eqref{eq:bound_y} with $q^t \equiv 0$, the Lipschitz continuity of the gradient of the cost function (cf. Assumption~\ref{assumption:lipschitz}) and~\eqref{eq:lambda_1}, to write
	$f(\avgxtm) - f(x_\star) \leq \tilde{\rho}^{2t}\frac{L\lambda_1^2}{2}\norm{y^0}_\ped^2
		\le
		\tilde{\rho}^{2t}\frac{L\lambda_1^2\lambda_2^2}{2}\norm{y^0}^2$,
	in which %
	we use~\eqref{eq:lambda_2}. The proof follows by using the definition~\eqref{eq:lambda} of $\lambda$.
	\end{document}